\documentclass[a4paper,11pt]{amsart}
\usepackage[T1]{fontenc}
\usepackage[english]{babel}
\usepackage[cp1252]{inputenc}
\usepackage{amsthm}
\usepackage{amsmath}
\usepackage{amsfonts}
\usepackage{tikz-cd}
\usepackage{pgfplots}
 \usepackage{microtype}
\usepackage[margin=0.8in]{geometry}

\usepackage{xcolor}
 \pgfplotsset{compat=1.18}
\usepgfplotslibrary{fillbetween}

\usepackage{amssymb}
\usepackage{hyperref}
\usepackage{color}
\usepackage{caption}
\usepackage{indentfirst}
\usepackage{amssymb}
\usepackage{eufrak}
\usepackage{mathrsfs}
\usepackage{xypic}

 \usepackage{booktabs,array,graphicx,longtable}
\usepackage{array,booktabs,tabularx}

\usepackage{booktabs}
\usepackage{listings}
\pgfplotsset{compat=1.18}
\definecolor{amoebablue}{RGB}{8,43,115}
\definecolor{contourred}{RGB}{122,0,25}

\theoremstyle{plain}   

\usepackage{booktabs}
\usepackage{graphicx}
\usepackage{array}

\definecolor{amoebadarkblue}{RGB}{0,63,127}
\definecolor{contourred}{RGB}{224,0,0}

\newif\ifhighdensityplot
\highdensityplotfalse

\definecolor{ContourColor}{RGB}{100,0,10}

\newcommand{\R}{\mathbb R}
\newcommand{\C}{\mathbb C}

\newcommand{\Log}{\mathrm{Log}\,}
\newcommand{\A}{\mathscr A}

\newcommand{\Cont}{\mathcal C}

\newcommand{\Crit}{\operatorname{Crit}}

\newcommand{\Newt}{\operatorname{Newt}}

\usepackage{array}
\usepackage{tabularx}

 \usepackage{caption}
 
\newcommand\restr[2]{{
  \left.\kern-\nulldelimiterspace 
  #1 
  \right|_{#2} 
  }}

\newtheorem*{mtheorem*}{Main Theorem}
\newtheorem{theorem}{Theorem}[section]
\newtheorem{definition}{Definition}[section]
\newtheorem{proposition}{Proposition}[section]
\newtheorem{corollary}{Corollary}[section]

\begin{document}
\title{Beyond Amoeba Boundaries} 

\author{Mounir Nisse}
\date{}

\address{Mounir Nisse\\
Department of Mathematics, Xiamen University Malaysia, Jalan Sunsuria, Bandar Sunsuria, 43900, Sepang, Selangor, Malaysia.
}
 
\email{\href{mailto:mounir.nisse@gmail.com}{mounir.nisse@gmail.com}, {mounir.nisse@xmu.edu.my}}
\urladdr{www.math.tamu.edu/\~{}nisse}
 
\thanks{This research is supported in part by Xiamen University Malaysia Research Fund (Grant no. XMUMRF/ 2024-C5/IMAT/0013).}

 

\maketitle

\begin{abstract}
We construct smooth plane curves whose logarithmic contours are smooth and embedded but strictly larger than their amoeba boundaries.  The examples include connected irreducible curves, disconnected unions of nonlinear components with two-dimensional amoebas, nested families with arbitrarily many components, and a uniform construction realizing every positive genus.  Exact critical-locus calculations distinguish the boundary-producing folds from compact or noncompact contour branches lying in the amoeba interior.  These constructions show that smoothness of the complete contour does not force contour--boundary equality and that internal critical-value branches persist under monomial transformations, controlled perturbations, and finite monomial coverings.
\end{abstract}

\section{Introduction}

The amoeba of an algebraic curve records the absolute values of its complex points while discarding their arguments.  More precisely, if $C\subset(\C^*)^2$ is the zero set of a Laurent polynomial, its amoeba is the closed subset $\A_C=\Log(C)\subset\R^2$, where $\Log(z,w)=(\log|z|,\log|w|)$.  Since their introduction by Gelfand, Kapranov, and Zelevinsky, amoebas have become a meeting point of complex algebraic geometry, real algebraic geometry, convex analysis, and tropical geometry \cite{GKZ,Mikhalkin}.  Their unbounded tentacles reflect the faces of the Newton polygon, their complement components are convex and carry the order map, and their area is governed by the Monge--Amp\`ere geometry of the Ronkin function \cite{FPT,PR,Rullgard}.  In the special extremal situation of maximal amoeba area, the geometry is tied to simple Harnack curves and to the rigid two-sheeted behavior of the logarithmic map \cite{Mikhalkin2000,MR2001,KOS}.

The present paper concerns a different and more local object: the complete logarithmic contour.  For a smooth curve $C=V(f)$, the logarithmic Gauss map is $\gamma_f:C\longrightarrow\mathbb P^1(\C)$, defined by $\gamma_f(z,w)=[zf_z(z,w):wf_w(z,w)]$.  The logarithmic critical locus is
$$
S_f=\Crit(\Log|_C)=\gamma_f^{-1}(\mathbb P^1(\R)),
$$
and the complete contour is its image $\Cont(A_C)=\Log(S_f)$.  This characterization, originating in the study of real toric curves and logarithmic Gauss maps, converts a differential criticality condition into an exact algebraic reality condition \cite{Mikhalkin2000,MadaniNisse}.  It also makes the contour accessible to elimination, real-root isolation, and direct parametrization.  The real degree of such contours and their intersections with affine lines were studied systematically by Lang, Shapiro, and Shustin, who emphasized that the contour, the amoeba boundary, and the logarithmic image of the real locus are distinct objects and may have substantially different geometry \cite{LSS2021}.

Although the inclusion $\partial\A_C\subseteq\Cont(\A_C)$ is automatic, equality need not hold.  The reason is already visible in the local normal form of a fold.  Near a nondegenerate logarithmic critical point, suitable real-analytic coordinates reduce the map to $(s,t)\mapsto(s,t^2)$.  The critical-value arc bounds the image of that particular local sheet, but it belongs to the boundary of the entire amoeba only when no other sheet covers the side locally omitted by the fold.  If another part of the complex curve maps regularly across the same target neighborhood, the fold remains in the complete contour while lying strictly inside the amoeba.  The distinction between a boundary-producing fold and an internal fold is therefore global in the source, even though criticality itself is detected locally.

This elementary observation leads to a natural existence problem.  How flexible can an amoeba contour be when it is required to remain smooth and embedded?  Can an internal critical-value branch occur without producing cusps, nodes, or self-tangencies?  Can this happen for a connected irreducible curve rather than only for a reducible union?  Can every algebraic component be nonlinear and have a two-dimensional amoeba?  Finally, can the phenomenon occur in arbitrarily high genus while retaining exact control of the logarithmic critical locus and its image?  The constructions developed here give affirmative answers to all of these questions.

The first family isolates the mechanism in its simplest exact form.  A rational graph has vertical logarithmic slices bounded by two explicit contour graphs, while an additional disjoint component maps to a central curve lying strictly between them.  Because the added contour is already covered on both sides by the amoeba of the graph component, it is internal.  Cyclotomic modifications replace the elementary graph by families controlled by $z^n-a^n$ and $z^n+a^n$, producing explicit half-width functions and allowing arbitrary monomial slopes.  The monomial component can then be replaced by a nonlinear hyperbola, so that no algebraic component is a line.  These examples establish the boundary--contour distinction by exact slice formulas rather than by a numerical picture.

The same idea admits a stronger nested form.  Two irreducible nonlinear graph components may be chosen so that both amoebas are two-dimensional and the amoeba of one component lies entirely in the interior of the other.  The contour of the inner component is then internal to the union amoeba, while the contour of the outer component produces the boundary.  A different construction uses the nonlinear level curves
$z+z^{-1}+w+w^{-1}=c$ with $c>4$.  Several distinct levels are pairwise disjoint in the algebraic torus, and their amoebas and contours are described by the exact inequalities involving $2\cosh x$ and $2\cosh y$.  Under an explicit separation condition, the complete union contour remains a smooth embedded real-analytic one-manifold.  This gives examples with arbitrarily many connected components and a parallelogram Newton polygon, while every component is nonlinear and has a two-dimensional amoeba.

Disconnected unions make the covering mechanism transparent, but they do not exhaust the phenomenon.  A central part of the paper is the connected irreducible translated conic
$$
(z-R)^2+(w-R)^2=1,
$$
with $R$ sufficiently large.  Its positive real circle maps to a compact logarithmic oval.  That oval is critical, yet it lies in the interior of the amoeba because a nonreal regular sheet covers a neighborhood on both sides.  The remaining critical components produce the remote boundary folds.  The proof combines the rational parametrization of the conic, the logarithmic Gauss criterion, quantitative control of a nonreal sheet, and the exclusion of ramification and off-diagonal critical-value coincidences.  Thus the compact internal oval is not the boundary of a bounded complement component; it is an additional component of the complete contour contained in the amoeba interior.

The circle is followed by two kinds of deformation.  A small anisotropic perturbation gives a connected noncircular conic for which the internal oval and the remote boundary components persist.  Unimodular monomial transformations then send these curves to algebraically different equations while acting on logarithmic space by invertible integral linear maps.  Such transformations preserve criticality, embeddedness, and the boundary--interior distinction.  In particular, they produce irreducible quartic equations whose Newton polygons and Euclidean shapes no longer resemble the original circle, even though the logarithmic geometry is transported exactly.

The final construction replaces the unimodular transformation by a finite monomial covering.  For every integer $n\geq2$, consider
$$
F_n(z,w)=(z^n-10)^2+(w-10)^2-1.
$$
The map $(z,w)\mapsto(z^n,w)$ is an unramified covering of the algebraic torus and intertwines the logarithmic maps by the linear transformation $(X,Y)\mapsto(nX,Y)$.  Consequently, the entire critical locus and the complete contour are pulled back from the translated conic, while their logarithmic images are transported by an invertible linear map.  The Newton polygon of $F_n$ is $\operatorname{conv}\{(0,0),(2n,0),(0,2)\}$, and the toric genus formula gives genus $n-1$ from its interior lattice points \cite{Khovanskii1978}.  Smoothness, irreducibility, Newton nondegeneracy, and connectedness are proved directly.  The construction therefore realizes every prescribed genus $g\geq1$ by taking $n=g+1$, and in every case the complete contour is smooth and embedded and contains a compact oval lying strictly inside the amoeba.

This covering argument also clarifies an important multiplicity issue.  Each critical point of the conic has several lifts, but points in one deck orbit have the same logarithmic image germ.  They must not be counted as distinct contour branches in an off-diagonal two-lift calculation.  After the deck graphs are removed, an off-deck coincidence upstairs would descend to a forbidden coincidence for the conic.  Thus the construction simultaneously controls ramification, image embeddedness, and toric-end separation.  The argument illustrates the broader principle that the topology of the critical source, the number of logarithmic lifts, and the geometry of the contour image must be kept separate.

The principal conclusion may be stated succinctly.  Smoothness and embeddedness of the complete contour do not imply contour--boundary equality.  The failure can be caused by compact internal ovals, internal noncompact graphs, or complete contours of nested components.  It occurs for connected and disconnected curves, for nonlinear components with two-dimensional amoebas, for Newton polygons of several combinatorial types, and in every positive genus.  These examples complement the degree-theoretic viewpoint of Lang, Shapiro, and Shustin \cite{LSS2021}: bounding how complicated a contour can be is only one side of the problem, while determining which smooth contour branches actually contribute to the amoeba boundary requires a global analysis of logarithmic sheets.

The paper is organized around increasingly strong versions of this principle.  The initial sections establish explicit disconnected graph, cyclotomic, hyperbolic, and nested constructions.  The next part develops families with arbitrarily many nonlinear components and analyzes their Newton polygons.  The connected part treats the translated circle, concentric conic unions, a noncircular perturbation, and unimodular monomial transforms.  The final sections give an irreducible transformed quartic and the finite-covering family $F_n$, culminating in the uniform existence theorem for every positive genus.  Throughout the paper, the blue regions in the figures represent the amoebas and the red curves represent the complete contours; the proofs, however, rely on exact parametrizations, critical equations, slice inequalities, and covering identities rather than on the drawings.
\medskip

\noindent{\it Acknowledgements.} The author is deeply grateful to Boris Shapiro for kindly sharing the joint paper with Lionel Lang and Eugenii Shustin \cite{LSS2021}; the present work was inspired by the challenging question raised in that paper.
Also, I am gratefully acknowledges the use of Codex AI, based on GPT-6, which used 
 Python, NumPy, SciPy, Matplotlib, Pillow, and LaTeX with TikZ and PGFPlots for the exact computations, numerical experiments, and figures presented in this work.  


\section{Definitions and the boundary--contour distinction}

Let $f\in\C[z^{\pm1},w^{\pm1}]$ and assume that $C=V(f)\subset(\C^*)^2$ is smooth.  The logarithmic Gauss map is $\gamma_f:C\to\mathbb{CP}^1$, defined by $\gamma_f(z,w)=[zf_z(z,w):wf_w(z,w)]$.  The standard criticality criterion is
$$
\Crit(\Log|_C)=\gamma_f^{-1}(\mathbb{RP}^1).
$$
This criterion is used repeatedly below.  It turns criticality into an exact reality condition and is the natural bridge between real algebraic geometry and amoeba contours \cite{Mikhalkin2000,LSS2021}.

If $C$ is disconnected, smoothness means smoothness as an algebraic submanifold of the torus.  Thus its irreducible components must be individually smooth and pairwise disjoint in $(\C^*)^2$.  For a factorization $f=f_1\cdots f_N$ with pairwise disjoint zero sets, one has $df=(\prod_{k\ne j}f_k)df_j$ along $V(f_j)$, so the union is smooth exactly when every component is smooth.

\begin{definition}\label{def:strict}
A smooth plane curve has the strict boundary--contour property when its complete contour is a smooth embedded real-analytic one-manifold and $\partial\A_C\subsetneq\Cont_C$.
\end{definition}

At an ordinary fold there are real-analytic source and target coordinates in which the logarithmic map has the local form $(s,t)\mapsto(s,t^2)$.  The critical-value arc bounds the image of that local sheet.  It is a boundary arc of the entire amoeba only when no other logarithmic sheet covers the locally missing side.  This local observation explains why smoothness of the contour alone cannot imply contour--boundary equality.

\subsection{A unified existence statement}

\begin{theorem}\label{thm:collected}%
There exist smooth connected irreducible plane curves and smooth disconnected plane curves having the strict boundary--contour property.  The connected examples may be chosen among the translated circles $C_R$, the noncircular conics $C_{R,\delta}$, and their unimodular monomial transforms.  The disconnected examples may have two or arbitrarily many connected components.  They may be chosen so that no component is a line and so that every component has a two-dimensional amoeba.
\end{theorem}

The proof occupies the remainder of the article.  Each family is established independently, so Theorem~\ref{thm:collected} is a summary rather than an additional existence argument.

\section{The elementary disconnected graph--diagonal family}

For $a>0$, put
$$
f_a(z,w)=\bigl((z+a)w-z(z-a)\bigr)(w-z)
$$
and let $C_a=V(f_a)\subset(\C^*)^2$.  Write $C_a^{(1)}$ for the graph of $w=z(z-a)/(z+a)$ and $C_a^{(2)}$ for the diagonal $w=z$.

\begin{theorem}\label{thm:elementary}
For every $a>0$, the curve $C_a=C_a^{(1)}\sqcup C_a^{(2)}$ is smooth and has exactly two connected irreducible components.  Its amoeba is the amoeba of $C_a^{(1)}$, its complete contour is smooth and embedded, and the diagonal $y=x$ is an internal contour component.  Consequently $C_a$ has the strict boundary--contour property.
\end{theorem}

\begin{proof}
The value $z=-a$ cannot occur on the first factor because substitution gives $-2a^2$.  Hence $C_a^{(1)}$ is a graph over $\C^*\setminus\{a,-a\}$, and its $w$-derivative $z+a$ never vanishes along the graph.  The graph is smooth and irreducible because its defining Laurent polynomial is primitive and has degree one in $w$.  The diagonal is smooth and irreducible.  Substitution of $w=z$ into the first factor gives $2az$, which does not vanish in the torus.  The components are therefore disjoint, and the product defines a smooth disconnected curve.

Put $z=re^{i\theta}$ and $x=\log r$.  The modulus of the graph coordinate satisfies
$\log|w|=x+\log|z-a|-\log|z+a|$.  Its extreme values on $|z|=r$ occur at $\sin\theta=0$.  With
$h_a(x)=\log((e^x+a)/|e^x-a|)$, the vertical amoeba slice is
$$
\{y:(x,y)\in\A_{C_a^{(1)}}\}=[x-h_a(x),x+h_a(x)].
$$
At $x=\log a$ the interval is understood as the whole real line.  The two graphs $y=x\pm h_a(x)$ form the boundary and the complete contour of the first component.  The diagonal component has amoeba and contour $y=x$.  Since $h_a(x)>0$ for every finite $x\ne\log a$, the diagonal lies in the interior of the graph amoeba.  Near $x=\log a$ the half-width tends to infinity, so every finite point on the diagonal is interior there as well.  Adding the diagonal does not enlarge the amoeba but adds an internal contour component.  The three image curves are mutually disjoint in the finite logarithmic plane and are smooth, proving the assertion.
\end{proof}

Figure~\ref{fig:elementary} uses $a=1$.

\begin{figure}[h]
\centering
\includegraphics[width=.28\textwidth]{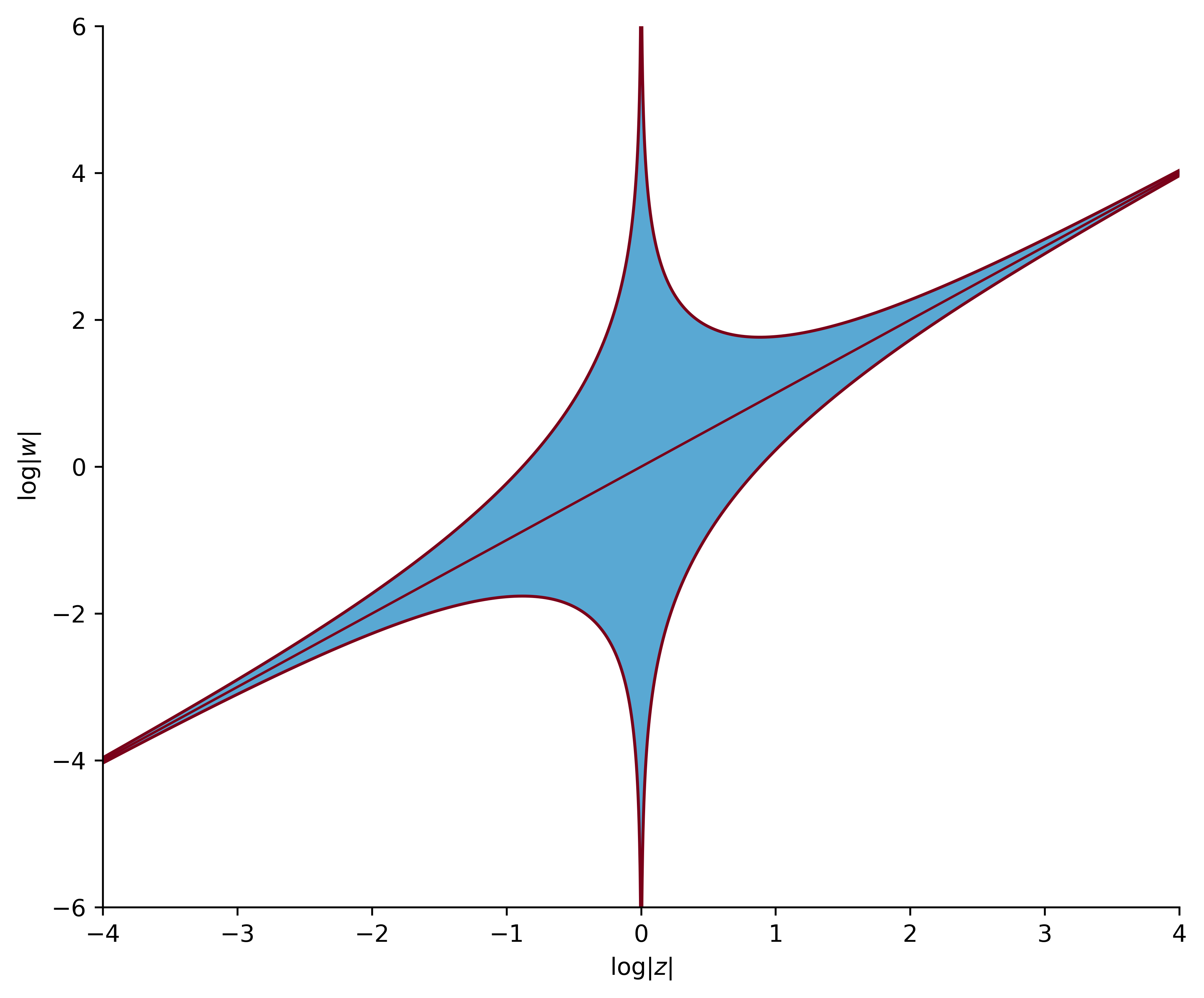}
\caption{The amoeba of $C_1$ in light blue and its complete contour in dark red.  The central red diagonal is internal.}
\label{fig:elementary}
\end{figure}

\subsection{The cyclotomic graph--monomial family}

Fix $n\ge2$, $m\ge1$, and $a>0$.  Put $A(z)=z^n+a^n$, $B(z)=z^n-a^n$, and define
$$
D_{n,m,a}=V\bigl((A(z)w-z^mB(z))(w-z^m)\bigr)\subset(\C^*)^2.
$$

\begin{theorem}\label{thm:cyclotomic}
The curve $D_{n,m,a}$ is smooth and has exactly two connected irreducible components.  Its complete contour is smooth and embedded.  The line $y=mx$ is an internal contour component, while the two graphs $y=mx\pm h_{n,a}(x)$ form the amoeba boundary, where
$h_{n,a}(x)=\log((e^{nx}+a^n)/|e^{nx}-a^n|)$.  Thus $D_{n,m,a}$ has the strict boundary--contour property.
\end{theorem}

\begin{proof}
The first component is the graph $w=z^mB/A$.  The factors $A$ and $z^mB$ are coprime in $\C[z^{\pm1}]$, so its defining polynomial is irreducible.  Its derivative with respect to $w$ is $A$, which cannot vanish on the graph because $A=0$ would leave the nonzero term $-z^mB$.  The second component is the smooth irreducible monomial curve $w=z^m$.  Restriction of the first factor to the second component gives $A z^m-z^mB=2a^nz^m$, so the components are disjoint.

For $z=re^{i\theta}$, direct differentiation gives
$$
\frac{\partial}{\partial\theta}\log\left|\frac{z^n-a^n}{z^n+a^n}\right|
=\frac{2na^nr^n(r^{2n}+a^{2n})\sin(n\theta)}{|z^n-a^n|^2|z^n+a^n|^2}.
$$
Hence the angular critical rays are $\sin(n\theta)=0$.  Their two critical-value graphs are $y=mx\pm h_{n,a}(x)$, and the complete vertical slice is the interval between them.  The monomial component contributes the line $y=mx$, which lies strictly between the two boundary graphs.  The divergence of $h_{n,a}$ at $x=\log a$ again implies interiority along the exceptional vertical fiber.  Smoothness and disjointness of the three contour images follow from $h_{n,a}>0$ and the nonvanishing of its derivative away from the exceptional end.
\end{proof}

The representative figure takes $(n,m,a)=(3,2,1)$.

\begin{figure}[h]
\centering
\includegraphics[width=.32\textwidth]{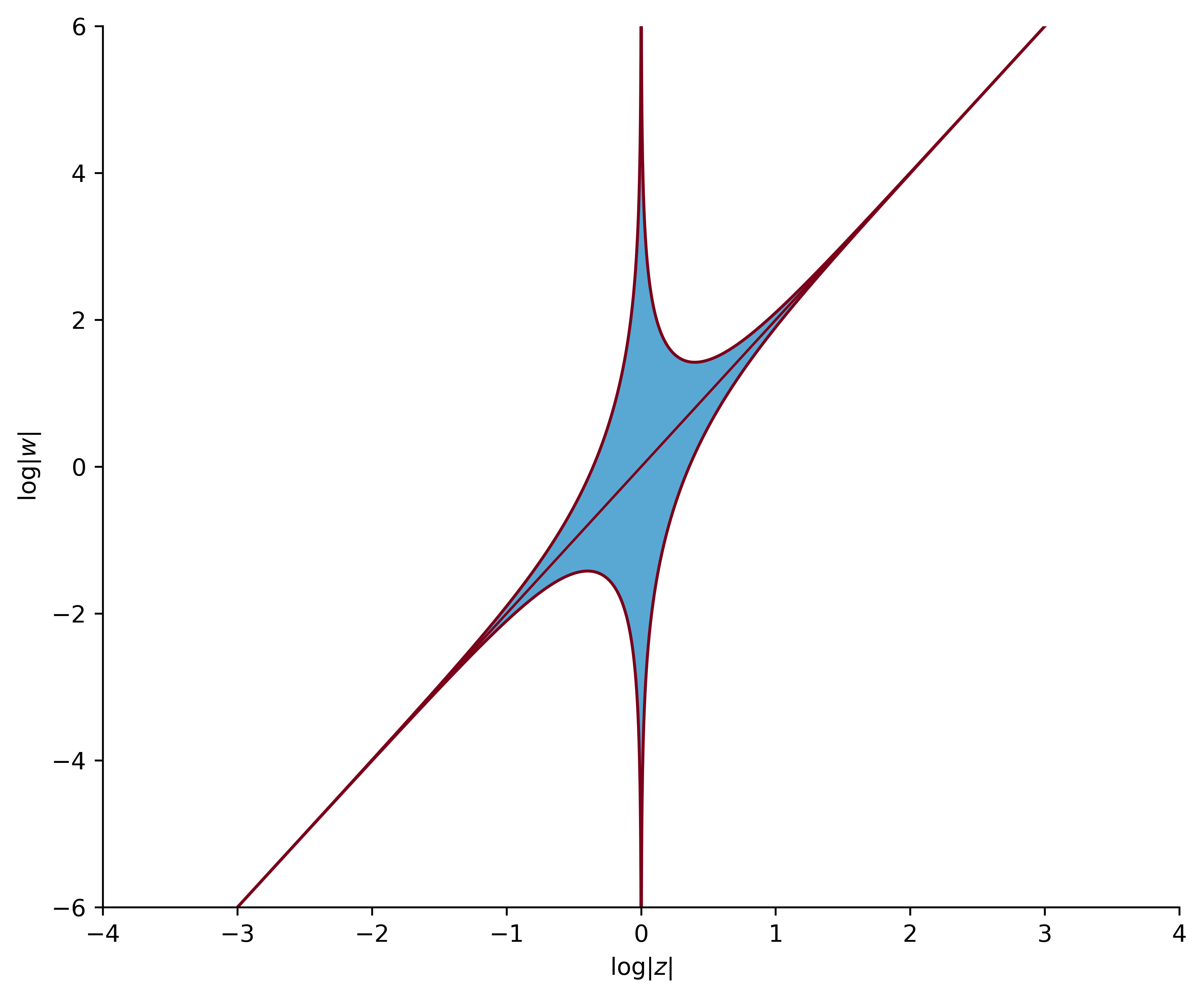}
\caption{The cyclotomic family for $(n,m,a)=(3,2,1)$.  The dark-red central line is internal.}
\label{fig:cyclotomic}
\end{figure}

\subsection{A disconnected family in which neither component is a line}

The preceding monomial component can be replaced by an irreducible hyperbola.  For $n\ge2$ and $a>0$, define
$$
E_{n,a}=V\bigl(((z^n+a^n)zw-(z^n-a^n))(zw-1)\bigr).
$$

\begin{theorem}\label{thm:hyperbola}
The curve $E_{n,a}$ is smooth and has two connected irreducible nonlinear components.  Its amoeba is two-dimensional, its complete contour is smooth and embedded, and the hyperbola contributes the internal contour line $y=-x$.  In particular, $E_{n,a}$ has the strict boundary--contour property.
\end{theorem}

\begin{proof}
The first component is the graph $w=B/(zA)$ and is irreducible because its equation is primitive and linear in $w$.  Its $w$-derivative is $zA$, which cannot vanish along the graph.  The second component $zw=1$ is a smooth irreducible conic in the affine plane and a smooth graph in the torus.  On $zw=1$, the first factor becomes $A-B=2a^n$, so the two components are disjoint.

The logarithmic slices of the first component are centered on $y=-x$ and have half-width $h_{n,a}(x)$.  The calculation in Theorem~\ref{thm:cyclotomic} gives its two boundary-producing contour graphs.  The hyperbola has amoeba and contour $y=-x$, which lies in the interior of every finite slice.  None of the two algebraic components is a line, and the complete contour is the disjoint union of the two outer graphs and the internal central graph.
\end{proof}

Figure~\ref{fig:hyperbola} uses $(n,a)=(4,1)$.

\begin{figure}[h]
\centering
\includegraphics[width=.32\textwidth]{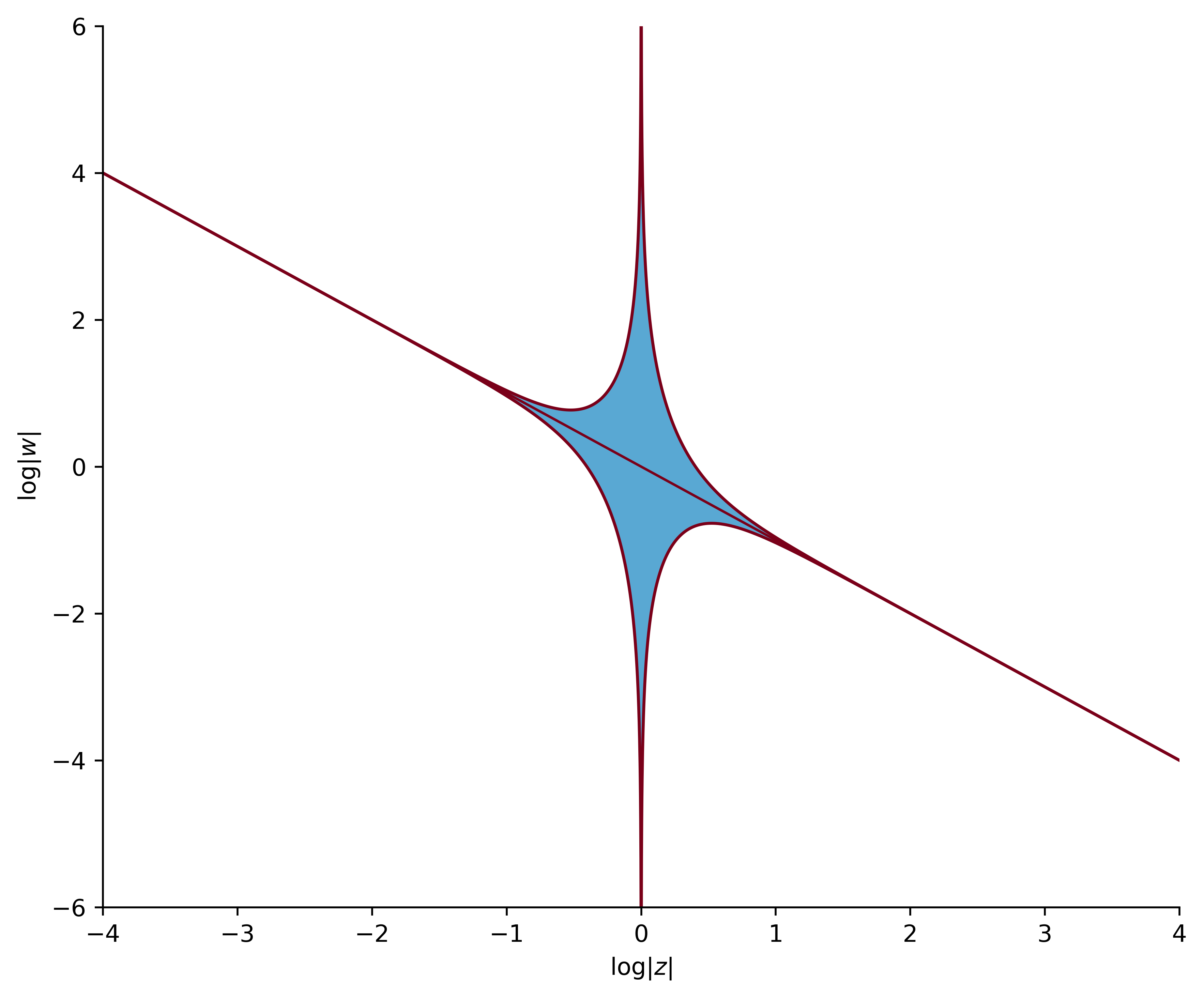}
\caption{The graph--hyperbola example.  Both algebraic components are nonlinear, and the central red line is internal.}
\label{fig:hyperbola}
\end{figure}

\section{Two nonlinear components with two-dimensional nested amoebas}

Let $A=z^n+a^n$ and $B=z^n-a^n$.  Define
$$
G_1=A w-z^mB,\qquad G_2=A^2w-z^mB^2,\qquad
C_{n,m,a}=V(G_1G_2).
$$

\begin{theorem}\label{thm:nested}
For every $n\ge2$, $m\ge1$, and $a>0$, the curve $C_{n,m,a}$ is smooth and has exactly two connected nonlinear irreducible components.  Both component amoebas are two-dimensional.  One has $\A_{V(G_1)}\subset\operatorname{Int}(\A_{V(G_2)})$, and therefore the complete contour of $V(G_1)$ is internal to the union amoeba.  The complete contour of the union is smooth and embedded.
\end{theorem}

\begin{proof}
Each component is a graph $w=z^m(B/A)^q$ with $q=1$ or $q=2$.  The coefficients $A^q$ and $-z^mB^q$ are coprime in the Laurent polynomial ring, proving irreducibility.  The derivative with respect to $w$ is $A^q$ and does not vanish on the graph, proving smoothness.  If the two graph values were equal, then $B/A=(B/A)^2$.  The value $B/A=0$ would give $w=0$, while $B/A=1$ would imply $a=0$.  Both are impossible in the stated torus and parameter range.  The components are disjoint.

The exact vertical slice of the $q$-th component is
$$
[mx-qh_{n,a}(x),mx+qh_{n,a}(x)].
$$
For $x\ne\log a$, the $q=1$ interval is strictly contained in the interior of the $q=2$ interval.  At the exceptional value, both intervals extend through every finite $y$, and the divergence of the outer half-width gives an open neighborhood of every point of the inner slice.  Thus the entire first amoeba is contained in the interior of the second.  The four contour graphs $y=mx\pm h_{n,a}(x)$ and $y=mx\pm2h_{n,a}(x)$ are pairwise disjoint and smooth in the finite plane.  The outer pair is the union boundary, and the inner pair is internal.
\end{proof}

Figure~\ref{fig:nested} uses $(n,m,a)=(3,2,1)$.

\begin{figure}[h]
\centering
\includegraphics[width=.32\textwidth]{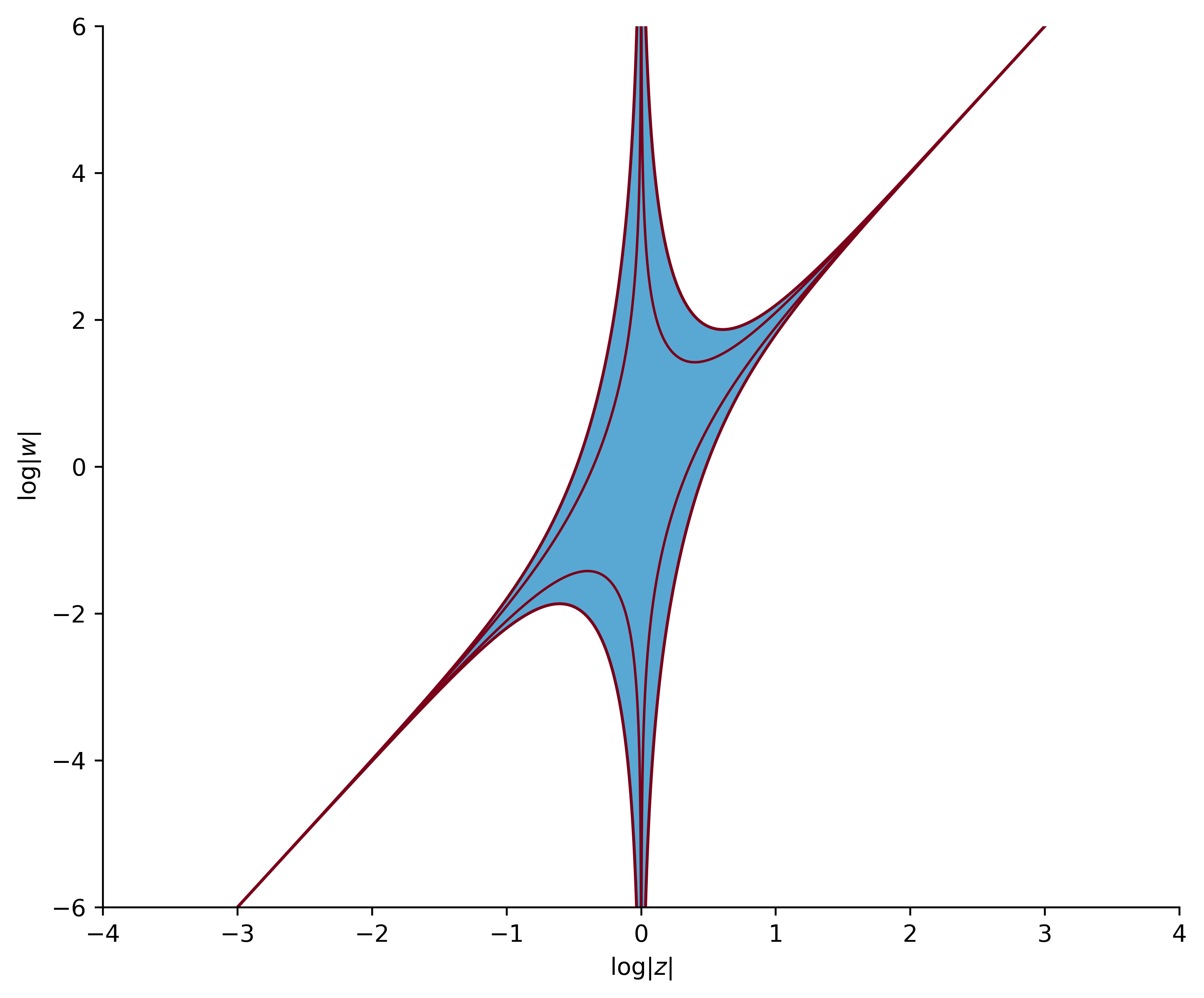}
\caption{Two two-dimensional nested amoebas.  The outer red pair is the boundary and the inner red pair is internal.}
\label{fig:nested}
\end{figure}

\subsection{Arbitrarily many nonlinear components with a parallelogram Newton polygon}

For $c>4$, put
$P_c(z,w)=z^2w+zw^2+z+w-czw$.  Dividing by $zw$ gives the equivalent Laurent equation $z+z^{-1}+w+w^{-1}=c$.  Let $4<c_1<\cdots<c_N$ and set $F_N=\prod_{j=1}^NP_{c_j}$ and $C_N=V(F_N)$.

\begin{theorem}\label{thm:levels}
Every component $V(P_{c_j})$ is smooth, connected, irreducible, and nonlinear, and every component amoeba is two-dimensional.  Distinct components are disjoint.  If $c_N-c_1<4$, then the complete contour of $C_N$ is a smooth embedded real-analytic one-manifold.  For $N\ge2$, suitable contour arcs are internal, and $C_N$ has the strict boundary--contour property.  The Newton polygon of $F_N$ is the parallelogram
$\operatorname{conv}\{(0,N),(N,0),(2N,N),(N,2N)\}$.
\end{theorem}

\begin{proof}
Writing $Q_c=z+z^{-1}+w+w^{-1}-c$, the equations $Q_c=Q_d=0$ would imply $c=d$.  Thus distinct levels are disjoint.  A singular point of $Q_c=0$ would satisfy $1-z^{-2}=1-w^{-2}=0$, so $z,w\in\{1,-1\}$ and $z+z^{-1}+w+w^{-1}\in\{-4,0,4\}$, contrary to $c>4$.  Hence every component is smooth.

As a quadratic in $w$, the cleared polynomial has discriminant
$(z^2-cz+1)^2-4z^2=(z^2-(c+2)z+1)(z^2-(c-2)z+1)$.  For $c>4$ this polynomial has four distinct nonzero roots and is not a square in $\C(z)$.  The quadratic is therefore irreducible over the Laurent polynomial ring.  Each component is connected and nonlinear.

Put $A(x)=2\cosh x$, $B(y)=2\cosh y$, $S=A+B$, and $D=|A-B|$.  The exact amoeba criterion is $D(x,y)\le c\le S(x,y)$.  The complete contour of one level is the disjoint union of
$$
A(x)+B(y)=c,\qquad A(x)-B(y)=c,\qquad B(y)-A(x)=c.
$$
These are regular embedded level curves because their gradients can vanish simultaneously only at $(0,0)$, where their levels are $4$, $0$, and $0$.  Contours of equal sign type and different levels are disjoint.  An intersection between $A+B=c_j$ and $A-B=c_k$ would force $2B=c_j-c_k\ge4$.  The assumption $c_N-c_1<4$ excludes this and all symmetric cases.  Thus the complete union contour is smooth and embedded.

The union amoeba criterion is that the interval $[D,S]$ contain at least one parameter $c_j$.  For strictness, take $N\ge2$ and choose an outer positive-real critical value of a larger level that is regularly covered by a lower level.  For the concrete levels $5,6,7$, let $4\cosh t=7$.  Then $(t,t)$ belongs to the contour of $C_7$.  Choosing $\theta$ with $\cos\theta=5/7$ gives $z=e^{t+i\theta}$ and $w=e^{t-i\theta}$ on $C_5$ with logarithmic value $(t,t)$.  The logarithmic Gauss ratio at this lift is nonreal, so the lift is regular and $(t,t)$ is interior to the union amoeba.  The same covering argument applies on an open set of parameter pairs inside the separation regime.  Hence the contour strictly contains the boundary.

Finally, $\Newt(P_c)=\operatorname{conv}\{(0,1),(1,0),(2,1),(1,2)\}$.  Newton polytopes add under products, yielding the stated parallelogram.
\end{proof}

\begin{figure}[h]
\centering
\includegraphics[width=.32\textwidth]{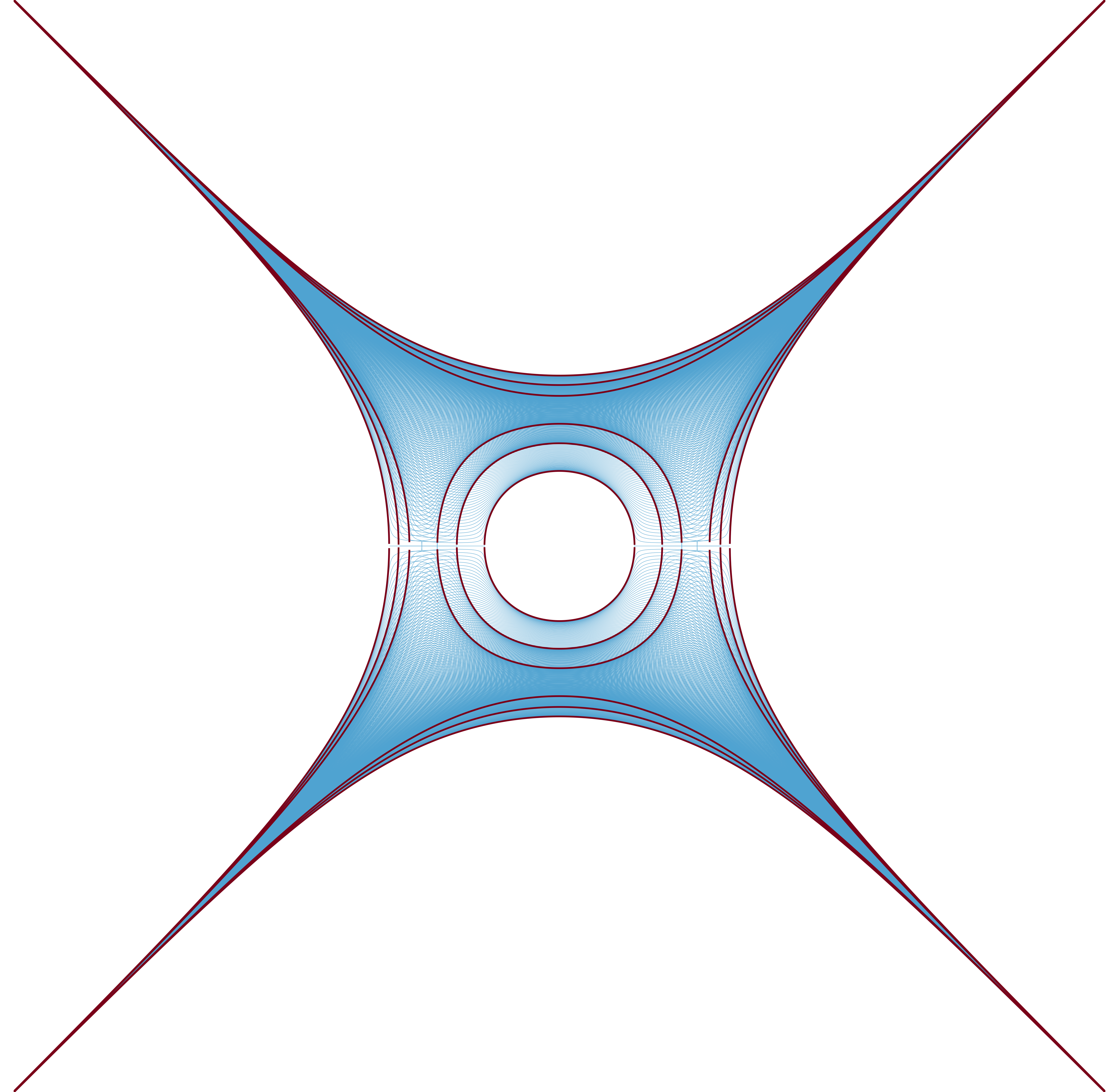}
\caption{The union amoeba for $c_1=5,c_2=6,c_3=7$ and the complete contours of all three nonlinear components.}
\label{fig:levels}
\end{figure}

\section{The connected irreducible circle}

The following example involving a circle was studied by G. Mikhalkin and O. Viro in an unpublished preprint, probably dating from 2002. I obtained it from J.-J. Risler in 2005.  For $R>1$, define
$C_R=V((z-R)^2+(w-R)^2-1)\subset(\C^*)^2$.

\begin{theorem}\label{thm:circle}
For every $R\ge10$, the curve $C_R$ is smooth, connected, and irreducible.  Its complete contour is a smooth embedded real-analytic one-manifold and
$$
\Cont_{C_R}=\partial\A_{C_R}\sqcup\Gamma_R,
$$
where $\Gamma_R=\{(\log(R+\cos\theta),\log(R+\sin\theta)):0\le\theta<2\pi\}$ is contained in $\operatorname{Int}(\A_{C_R})$.
\end{theorem}

\begin{proof}
The gradient can vanish only at $(R,R)$, which does not lie on the conic.  Its projective closure is a nonsingular conic, hence it is irreducible, and deleting finitely many toric boundary points leaves a connected Riemann surface.  A rational parametrization is
$$
z_R(\zeta)=R+\frac12(\zeta+\zeta^{-1}),\qquad
w_R(\zeta)=R+\frac{1}{2i}(\zeta-\zeta^{-1}).
$$
The positive real circle corresponds to $|\zeta|=1$.  Its logarithmic image $\Gamma_R$ is embedded because the coordinatewise logarithm is a diffeomorphism on the positive quadrant.  Its logarithmic Gauss map is real, so $\Gamma_R$ is part of the contour.

To prove internality, put $\zeta=R\xi$ and write $\xi=\alpha+i\beta$.  After subtracting $(\log R,\log R)$, the logarithmic map converges in $C^1$ on compact sets to the coordinatewise logarithm of
$A=(1+\alpha/2)^2+\beta^2/4$ and $B=(1+\beta/2)^2+\alpha^2/4$.  The Jacobian determinant of $(A,B)$ is $1+(\alpha+\beta)/2$.  On the rational box $[-5/2,-3/2]^2$ it belongs to $[-3/2,-1/2]$.  For $R\ge10$, direct rational interval evaluation bounds the derivative error by $9/100$, so the exact determinant remains separated from zero by $41/100$.  The image of the smaller box $[-9/4,-7/4]^2$ contains the target square $[-1/8,1/8]^2$.  Meanwhile
$|\log((R+\cos\theta)/R)|<1/8$ and $|\log((R+\sin\theta)/R)|<1/8$.  Thus a regular nonreal sheet covers a neighborhood of the entire real logarithmic oval, proving $\Gamma_R\subset\operatorname{Int}(\A_{C_R})$.

Writing $\zeta=u+iv$, the exact critical equation factors as $(u^2+v^2-1)H_R(u,v)=0$.  The first factor gives the internal real oval.  The remaining factor gives the remote folds.  Compactified interval evaluation gives normalized lower bounds $1/50$ for the critical-source gradient, $1/200$ for the contour tangent, and $1/500$ for the saturated off-diagonal two-lift separation when $R\ge10$.  Reciprocal charts separate the toric ends.  Hence the remote critical components map to disjoint embedded definite folds and exhaust the amoeba boundary.  This proves the asserted disjoint decomposition.
\end{proof}

The figure uses $R=10$.  Every red branch belongs to the complete contour, including the small internal oval.

\begin{figure}[h]
\centering
\includegraphics[width=.32\textwidth]{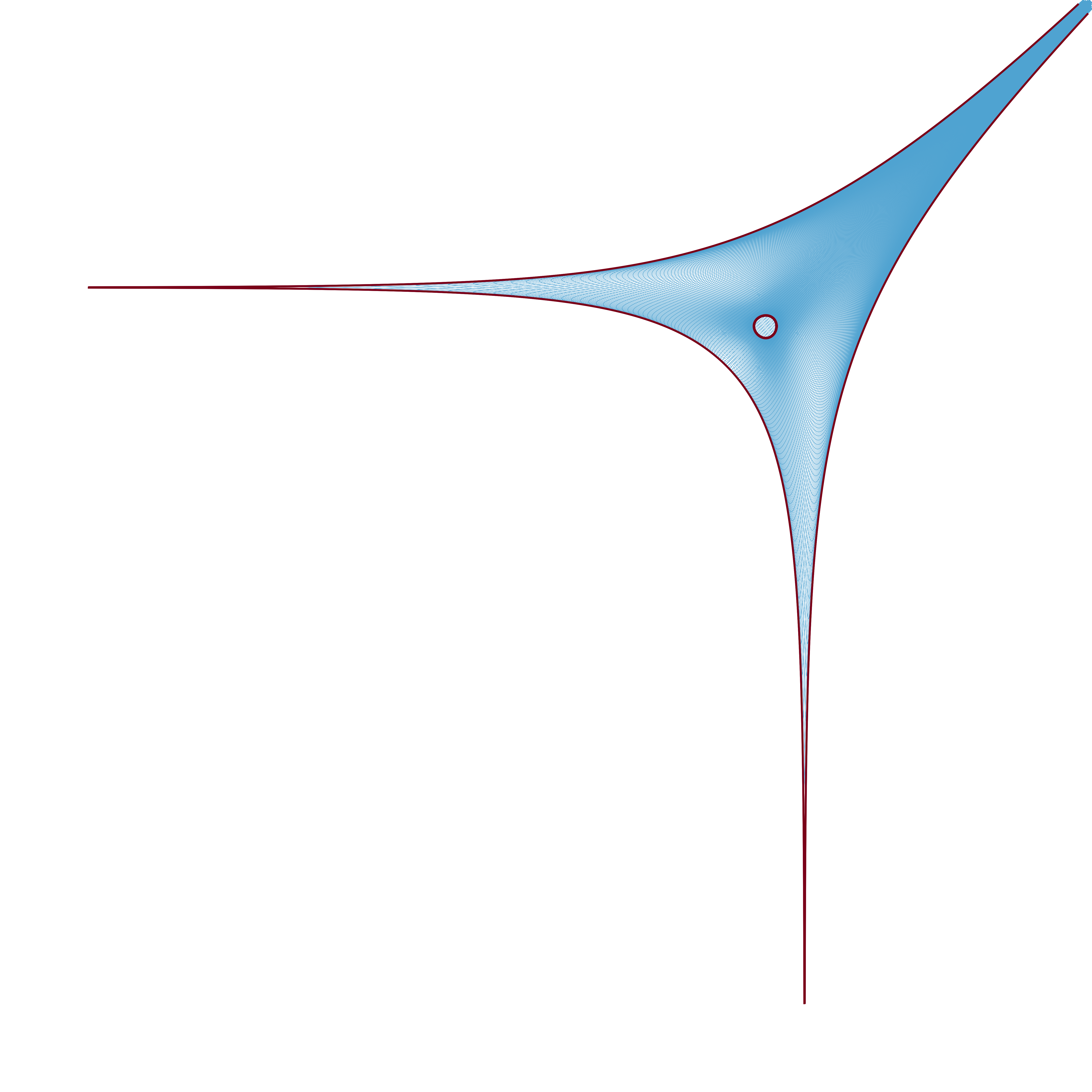}
\caption{The amoeba of the irreducible translated circle in light blue and its complete contour in dark red.}
\label{fig:circle}
\end{figure}


\subsection{Two and three concentric conics with internal contour ovals}
\label{subsec:two-three-concentric-conics}

The translated-circle construction extends naturally to disconnected curves formed by two or three concentric complex conics.  These examples exhibit several internal contour ovals while keeping every algebraic component smooth, irreducible, and nonlinear.  They also illustrate why smoothness of the source curve, smoothness of each component contour, and smoothness of the complete contour of the union are logically different assertions.  The complete contour is the full critical-value set $\Cont(\A_C)=\Log(\Crit(\Log|_C))$, as in the real-degree framework of Lang, Shapiro, and Shustin \cite{LSS2021}.  It always contains $\partial\A_C$, but it may also contain critical branches in $\operatorname{Int}(\A_C)$.  The distinction between singularities of the critical source and singularities of its logarithmic image is central to the analytic theory of contours \cite{NisseAnalyticity,NisseSing}.

Fix $R>0$ and define $f_{R,r}(z,w)=(z-R)^2+(w-R)^2-r^2$ for $0<r<R$.  The positive real locus of $V(f_{R,r})$ is the Euclidean circle with center $(R,R)$ and radius $r$, whereas $V(f_{R,r})\subset\C^2$ is a complex affine conic.  The amoeba is the logarithmic image of the full complex curve and not merely the logarithmic image of its real circle.  The Newton polygon of a factor is $\Delta_2=\operatorname{conv}\{(0,0),(2,0),(0,2)\}$, and its projective closure is a nonsingular rational conic.

The gradient is $df_{R,r}=(2(z-R),2(w-R))$.  Its only possible zero is $(R,R)$, which is not on the curve because $f_{R,r}(R,R)=-r^2$.  Hence every factor curve is smooth.  If $r\ne\rho$, the equations $f_{R,r}=f_{R,\rho}=0$ have no common solution, since subtraction would give $r^2=\rho^2$.  Products formed from distinct radii therefore define disjoint unions of smooth conics.

For a smooth plane curve $V(g)$, the logarithmic Gauss map is $\gamma_g(z,w)=[zg_z:wg_w]$.  A point of $V(g)$ is critical for the restriction of $\Log$ precisely when $\gamma_g(z,w)\in\mathbb P^1(\R)$ \cite{Mikhalkin2000}.  For the circle factor, $\gamma_{R,r}(z,w)=[z(z-R):w(w-R)]$, and the criticality equation is
$$
\operatorname{Im}\left(z(z-R)\overline{w(w-R)}\right)=0.
$$
If $F$ is a product of distinct factors, then on $V(f_{R,r_j})$ both logarithmic derivatives of $F$ acquire the same nonzero factor $\prod_{k\ne j}f_{R,r_k}$.  This factor cancels in projective coordinates.  Thus the logarithmic critical locus of the product is the disjoint union of the critical loci of its components, while the complete contour is the set-theoretic union of their critical-value images.

The radius dependence is reduced exactly to the unit-radius model.  Substituting $z=rZ$ and $w=rW$ gives $f_{R,r}(rZ,rW)=r^2((Z-R/r)^2+(W-R/r)^2-1)$.  Hence the torus dilation $D_r(Z,W)=(rZ,rW)$ identifies $V(f_{R/r,1})$ with $V(f_{R,r})$.  Since $\Log\circ D_r=\Log+(\log r,\log r)$, one obtains $\A_{R,r}=\A_{R/r,1}+(\log r,\log r)$ and $\Cont(\A_{R,r})=\Cont(\A_{R/r,1})+(\log r,\log r)$.  Smoothness, ramification, self-intersection behavior, and the internal or boundary-producing character of an individual contour branch are preserved by this translation.

The positive real circle has the parametrization $z=R+r\cos\theta$ and $w=R+r\sin\theta$.  Because $r<R$, both coordinates remain positive.  Its logarithmic image is
$$
\Gamma_{R,r}=\{(\log(R+r\cos\theta),\log(R+r\sin\theta)):0\le\theta<2\pi\}.
$$
The coordinatewise logarithm is a diffeomorphism from the positive quadrant to $\R^2$, so $\Gamma_{R,r}$ is a smooth embedded oval.  It is contained in the complete contour because the logarithmic Gauss map is real along the positive real circle.  For the unit circle translated to a center $S\ge10$, the affine and reciprocal critical-locus calculations certify that this real oval has no ramification and is covered by another logarithmic sheet.  It therefore lies in the interior of the amoeba.  Transport by $D_r$ yields $\Gamma_{R,r}\subset\operatorname{Int}(\A_{R,r})$ whenever $R/r\ge10$.  This conclusion depends on the critical-locus and sheet calculation, rather than on the appearance of a plot; this is precisely the distinction required in the study of analytic contour images \cite{NisseAnalyticity,NisseSing}.

\paragraph{The two-circle example.}
Let $0<r_1<r_2<R$ and put $F_{R;r_1,r_2}=f_{R,r_1}f_{R,r_2}$.  The curve $C_{R;r_1,r_2}=V(F_{R;r_1,r_2})\cap(\C^*)^2$ is the disjoint union of two irreducible conics.  At a point of the $j$-th component, $dF_{R;r_1,r_2}=f_{R,r_k}df_{R,r_j}$ with $f_{R,r_k}=r_j^2-r_k^2\ne0$, so the product curve is smooth.  Its Newton polygon is the Minkowski sum $2\Delta_2=\operatorname{conv}\{(0,0),(4,0),(0,4)\}$, although the normalization of each component has genus zero.

\begin{proposition}[Two internal contour ovals]
\label{prop:two-circle-ovals}
If $0<r_1<r_2\le R/10$, then $C_{R;r_1,r_2}$ is a smooth disconnected curve with two nonlinear irreducible components.  Its complete contour contains two pairwise disjoint smooth embedded ovals $\Gamma_{R,r_1}$ and $\Gamma_{R,r_2}$ in the interior of its amoeba.  Consequently, $\partial\A_{C_{R;r_1,r_2}}\subsetneq\Cont(\A_{C_{R;r_1,r_2}})$.
\end{proposition}

\begin{proof}
The algebraic assertions follow from the gradient and disjointness calculations above.  Since $r_j\le r_2\le R/10$, one has $R/r_j\ge10$ for both components.  The unit-circle certificate transported by $D_{r_j}$ places $\Gamma_{R,r_j}$ in $\operatorname{Int}(\A_{R,r_j})$.  The amoeba of the union is $\A_{R,r_1}\cup\A_{R,r_2}$, and any open set contained in one component amoeba is contained in the union.  Both ovals therefore lie in the interior of the union amoeba.  They belong to the complete contour and are disjoint because they are the images, under the injective coordinatewise logarithm on the positive quadrant, of two disjoint Euclidean circles.  Thus the contour contains points outside the amoeba boundary.
\end{proof}

For $R=10$, $r_1=1/2$, and $r_2=1$, the polynomial is
$$
F_2(z,w)=\left((z-10)^2+(w-10)^2-\frac14\right)
\left((z-10)^2+(w-10)^2-1\right).
$$
The effective unit centers are $20$ and $10$, so Proposition~\ref{prop:two-circle-ovals} applies.  Figure~\ref{fig:two-circle-subsection} displays the union amoeba as a non-flat blue line cloud and the sampled complete contour in dark red.  White remains visible between the traces, and their overlap gives local variation in blue density.  The two small nested red ovals are internal contour branches.

\paragraph{The three-circle example.}
Let $0<r_1<r_2<r_3<R$ and put $F_{R;r_1,r_2,r_3}=f_{R,r_1}f_{R,r_2}f_{R,r_3}$.  Its zero set is a smooth disjoint union of three irreducible conics.  The Newton polygon of the product is $3\Delta_2=\operatorname{conv}\{(0,0),(6,0),(0,6)\}$.  As in the two-factor case, the interior lattice points of this product polygon do not give the genera of the three separate normalizations.

\begin{proposition}[Three internal contour ovals]
\label{prop:three-circle-ovals}
If $0<r_1<r_2<r_3\le R/10$, then $C_{R;r_1,r_2,r_3}$ is a smooth disconnected curve with three nonlinear irreducible components.  Its complete contour contains three pairwise disjoint smooth embedded ovals $\Gamma_{R,r_1}$, $\Gamma_{R,r_2}$, and $\Gamma_{R,r_3}$ in the interior of its amoeba.  Consequently, $\partial\A_{C_{R;r_1,r_2,r_3}}\subsetneq\Cont(\A_{C_{R;r_1,r_2,r_3}})$.
\end{proposition}

\begin{proof}
The proof is the same componentwise argument with three radii.  The inequalities give $R/r_j\ge10$ for every $j$, so the translated unit-circle certificate places each oval inside its component amoeba.  Since $\A_{C_{R;r_1,r_2,r_3}}=\bigcup_{j=1}^3\A_{R,r_j}$, all three ovals lie in the interior of the union.  They are logarithmic critical-value components and are pairwise disjoint by injectivity of the coordinatewise logarithm on the positive quadrant.  Hence the boundary is strictly contained in the complete contour.
\end{proof}

For $R=10$, $r_1=2/5$, $r_2=7/10$, and $r_3=1$, the defining polynomial is
$$
\begin{aligned}
F_3(z,w)={}&\left((z-10)^2+(w-10)^2-\frac4{25}\right)\\
&\cdot\left((z-10)^2+(w-10)^2-\frac{49}{100}\right)
\left((z-10)^2+(w-10)^2-1\right).
\end{aligned}
$$
The effective unit centers are $25$, $100/7$, and $10$, so Proposition~\ref{prop:three-circle-ovals} applies.  Figure~\ref{fig:three-circle-subsection} uses the same non-flat blue texture and dark-red contour convention as the two-circle figure.  The three small nested red ovals are internal.

\begin{figure}[h]
\centering
\includegraphics[width=.32\textwidth]{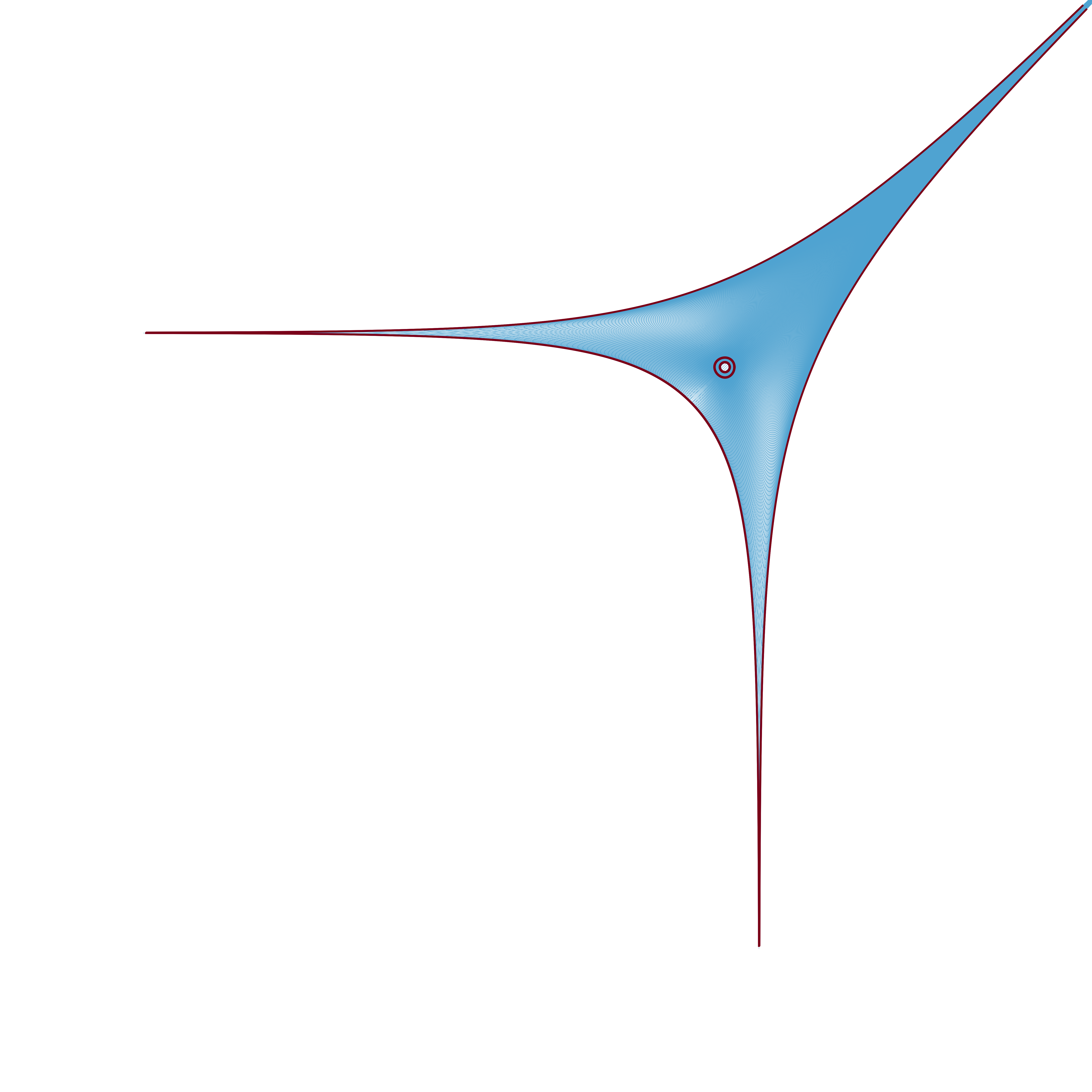}\qquad 
\includegraphics[width=.32\textwidth]{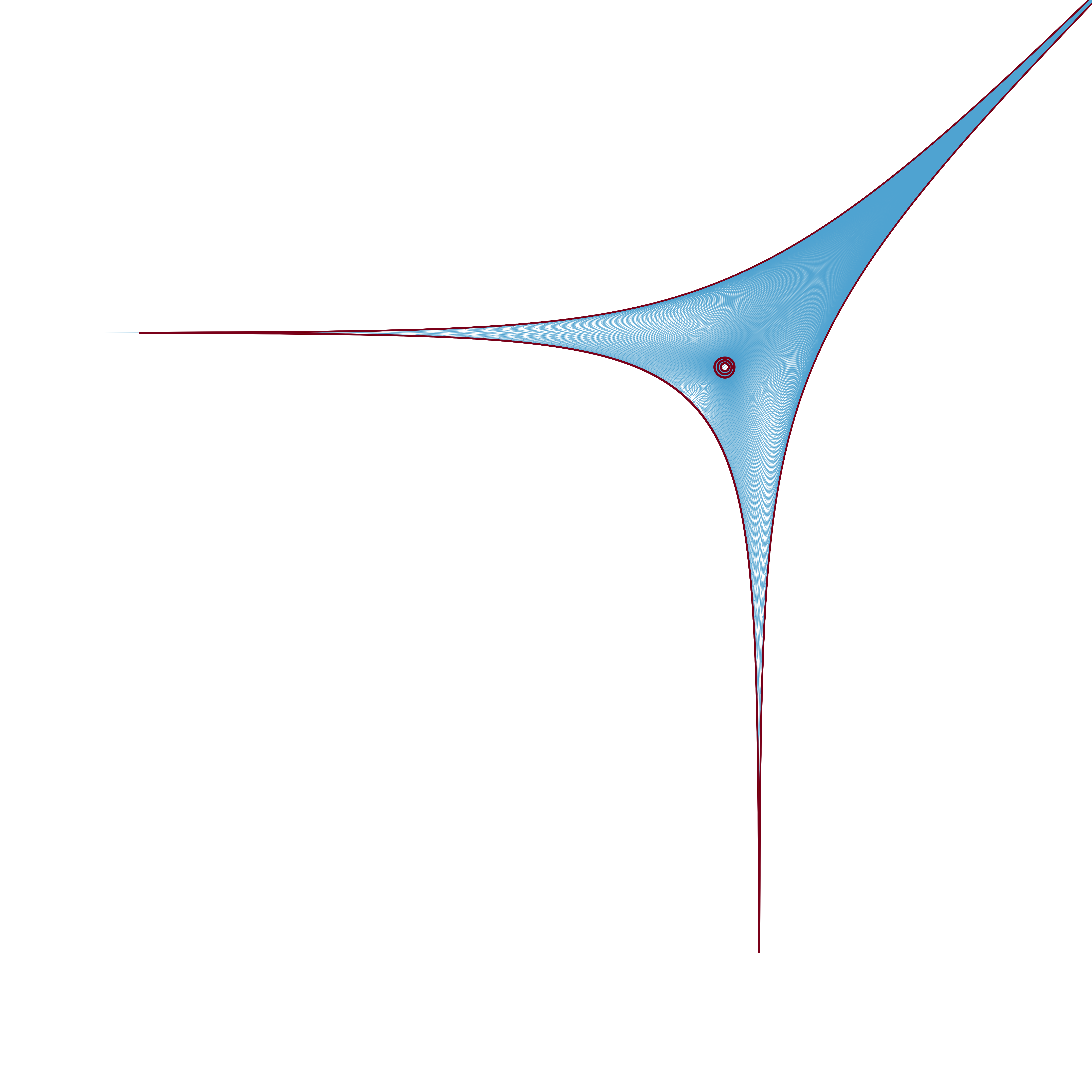}
\caption{On the left the two-circle amoeba for $R=10$, $r_1=1/2$, and $r_2=1$, shown with a non-flat blue texture, together with its sampled complete contour in dark red.  The two compact nested red ovals lie in the interior of the amoeba.
On the right the three-circle amoeba for $R=10$ and radii $2/5$, $7/10$, and $1$, shown as a non-flat blue line cloud, together with its sampled complete contour in dark red.  The three compact nested red ovals lie in the amoeba interior.}
\label{fig:three-circle-subsection}
\end{figure}

\paragraph{Smoothness status of the complete union contours.}
Propositions~\ref{prop:two-circle-ovals} and \ref{prop:three-circle-ovals} prove algebraic smoothness, componentwise smooth embedded internal ovals, and strict boundary--contour inclusion.  They do not by themselves prove that the complete contour of either disconnected union is globally embedded.  Two points on different conic components can have identical coordinate moduli even though the points and the algebraic components are disjoint.

For two radii $r_j\ne r_k$, one must introduce critical lifts $(z,w)\in C_{R,r_j}$ and $(Z,W)\in C_{R,r_k}$ and impose $|z|^2=|Z|^2$ and $|w|^2=|W|^2$.  In real coordinates, this produces the two real equations for each conic, the two real logarithmic criticality equations, and the two equal-modulus equations.  The system must be saturated by $zwZW$ and analyzed in all sixteen ordered pairs of the four toric chart types.  A nonzero determinant of the two logarithmic image tangents certifies a transverse contour crossing; when the determinant vanishes, higher logarithmic jets are needed. %


\subsection{A connected noncircular elliptic perturbation}

For real $\delta$, put
$$f_{R,\delta}=(z-R)^2+(w-R)^2+\delta(z-R)(w-R)-1, \quad \textrm{and} \,\, C_{R,\delta}=V(f_{R,\delta})$$

\begin{theorem}\label{thm:ellipse}
For every $R\ge10$ and $0<|\delta|\le1/10$, the curve $C_{R,\delta}$ is smooth, connected, irreducible, and noncircular.  Its complete contour is a smooth embedded real-analytic one-manifold equal to the disjoint union of its amoeba boundary and the internal logarithmic image of its positive real ellipse.
\end{theorem}

\begin{proof}
Put $X=z-R$ and $Y=w-R$.  The quadratic form $X^2+Y^2+\delta XY$ has eigenvalues $1\pm\delta/2$, which lie in $[19/20,21/20]$.  Its real level set is therefore a noncircular ellipse.  The gradient equations force $X=Y=0$ when $|\delta|<2$, which is incompatible with the level equation.  The projective conic is nonsingular, hence irreducible, and its torus part is connected.

A rational parametrization is obtained from $D(t)=t^2+\delta t+1$ by
$X=(t^2-1)/D$ and $Y=-(\delta t^2+2t)/D$.  The real projective parameter line produces the positive real ellipse and is logarithmically critical.  Its image is embedded.  The semiaxes are at most $\sqrt{20/19}<103/100$, so for $R\ge10$ the logarithmic ellipse remains inside the translated square of radius $1/8$ used in the proof of Theorem~\ref{thm:circle}.

The remote logarithmic Jacobian varies analytically with $\delta$.  Rational interval evaluation on $R\ge10$ and $|\delta|\le1/10$ gives the lower bound $7/25$.  Thus the regular remote sheet continues to cover a neighborhood of the real logarithmic ellipse.  The compactified critical-source, contour-tangent, two-lift, and toric-end bounds remain strict on this closed parameter region because $4-\delta^2\ge399/100$.  No new critical component, ramification, self-intersection, or end collision appears.  The nonreal folds remain boundary-producing, while the real logarithmic ellipse remains internal.
\end{proof}

Figure~\ref{fig:ellipse} uses $(R,\delta)=(10,1/10)$.

\begin{figure}[h]
\centering
\includegraphics[width=.32\textwidth]{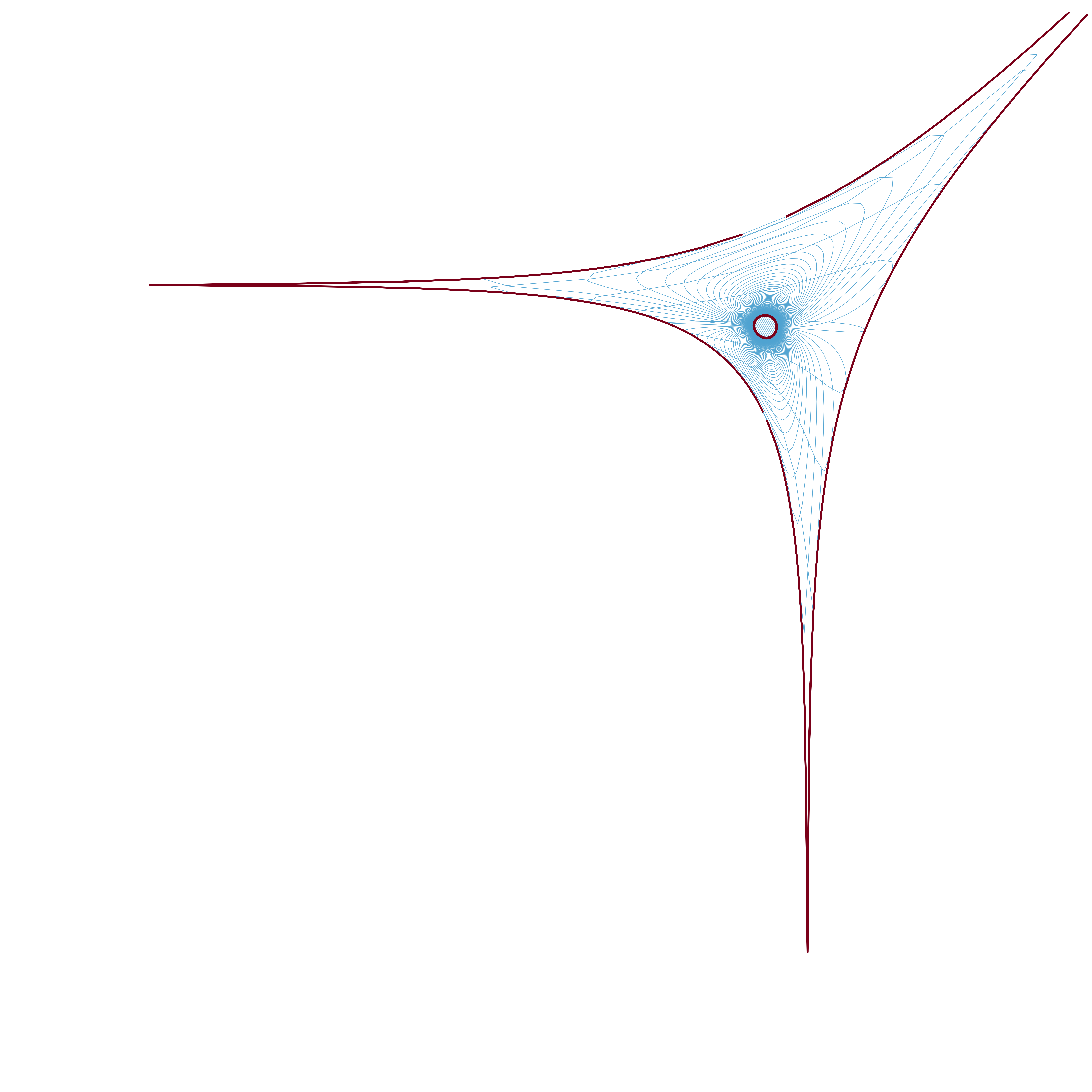}
\caption{The amoeba of the noncircular irreducible conic in light blue and its complete contour in dark red.}
\label{fig:ellipse}
\end{figure}

\subsection{A monomially transformed irreducible conic}

The connected examples remain valid under unimodular monomial automorphisms.  Consider
$$
K_R=V((z-R)^2+(zw-R)^2-1).
$$

\begin{corollary}\label{cor:monomial}
For $R\ge10$, the curve $K_R$ is smooth, connected, and irreducible, and it has the strict boundary--contour property.
\end{corollary}

\begin{proof}
The torus automorphism $\Phi(z,w)=(z,zw)$ has inverse $(u,v)\mapsto(u,v/u)$.  It identifies $K_R$ with $C_R$.  In logarithmic coordinates it induces the invertible linear map $(x,y)\mapsto(x,x+y)$.  Consequently the amoeba, contour, boundary, smoothness, embeddedness, and strict inclusion are transported by a linear diffeomorphism.  The displayed equation is not an affine circle or ellipse in the variables $(z,w)$, although it is torically equivalent to the circle.
\end{proof}

\begin{figure}[h]
\centering
\includegraphics[width=.28\textwidth]{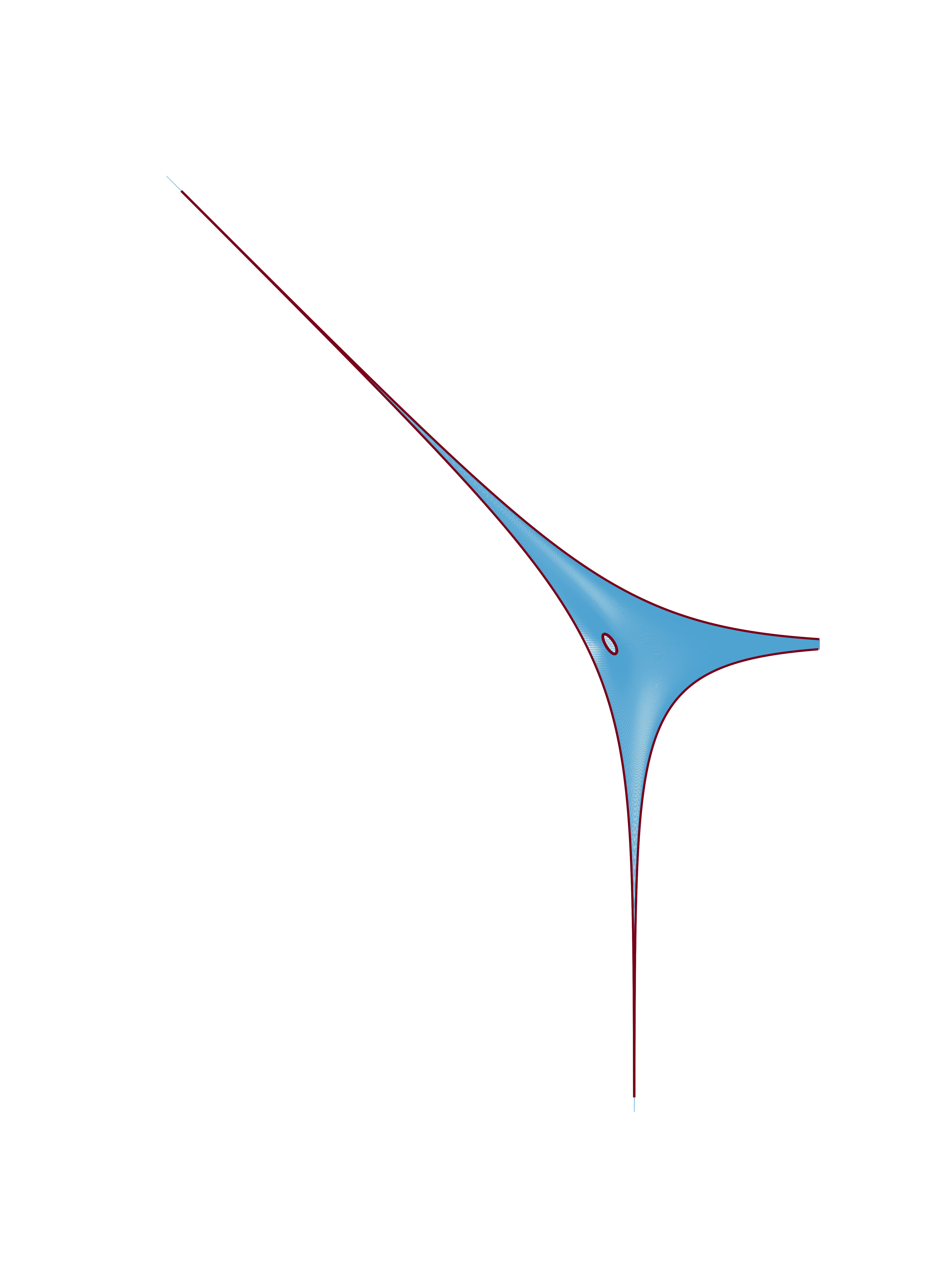} 
\caption{The amoeba and complete contour of the monomially transformed irreducible conic $K_{10}$.}
\label{fig:monomial}
\end{figure}

 \section{A genus-one curve}

The level curve
$$
C^{(1)}=\{z+z^{-1}+w+w^{-1}=5\}\subset(\C^*)^2
$$
provides a positive-genus control.  Clearing denominators gives $H=z^2w+zw^2+z+w-5zw$, whose Newton polygon is the diamond with vertices $(0,1),(1,0),(2,1),(1,2)$.  The point $(1,1)$ is its unique interior lattice point.

\begin{theorem}\label{thm:genus-one-control}%
The curve $C^{(1)}$ is smooth, irreducible, Newton nondegenerate, and of genus one.  Its logarithmic critical locus and complete contour are smooth, the contour is embedded, and
$\Cont_{C^{(1)}}=\partial\A_{C^{(1)}}$.
It is not monomially equivalent to an affine circle or ellipse.
\end{theorem}

\begin{proof}
The logarithmic derivatives are $z-z^{-1}$ and $w-w^{-1}$.  A singular point would have $z,w\in\{1,-1\}$, but the level would then be $-4$, $0$, or $4$, never $5$.  Every face polynomial is a primitive binomial, so Newton nondegeneracy follows.  The toric genus formula gives genus one.  Irreducibility follows because a factorization would split the Newton polygon and force intersecting compactified components, contradicting smoothness.

Put $p=z+z^{-1}$ and $q=w+w^{-1}$.  On the curve $p+q=5$.  The reality condition for the logarithmic Gauss ratio forces $p,q$ to be real and then forces $z,w$ to be real; otherwise one logarithmic derivative would be real and the other purely imaginary.  Hence the critical locus is exactly the real part of the curve.  Its contour consists of the three regular level sets
$$
2\cosh x+2\cosh y=5,\qquad
2\cosh x-2\cosh y=5,\qquad
2\cosh y-2\cosh x=5.
$$
They are smooth, embedded, and pairwise disjoint.  The exact amoeba criterion is $|2\cosh x-2\cosh y|\le5\le2\cosh x+2\cosh y$, so all three level sets form the boundary.  The unique interior lattice point is invariant under integral affine transformations of the exponent lattice, whereas the Newton triangle of an affine conic has none.  Thus no allowed monomial equivalence identifies this curve with a circle or ellipse.  This total-real logarithmic behavior is the genus-one simple-Harnack situation described in \cite{MR2001,KOS2006}.
\end{proof}

\begin{figure}[h]
\centering
\includegraphics[width=.32\textwidth]{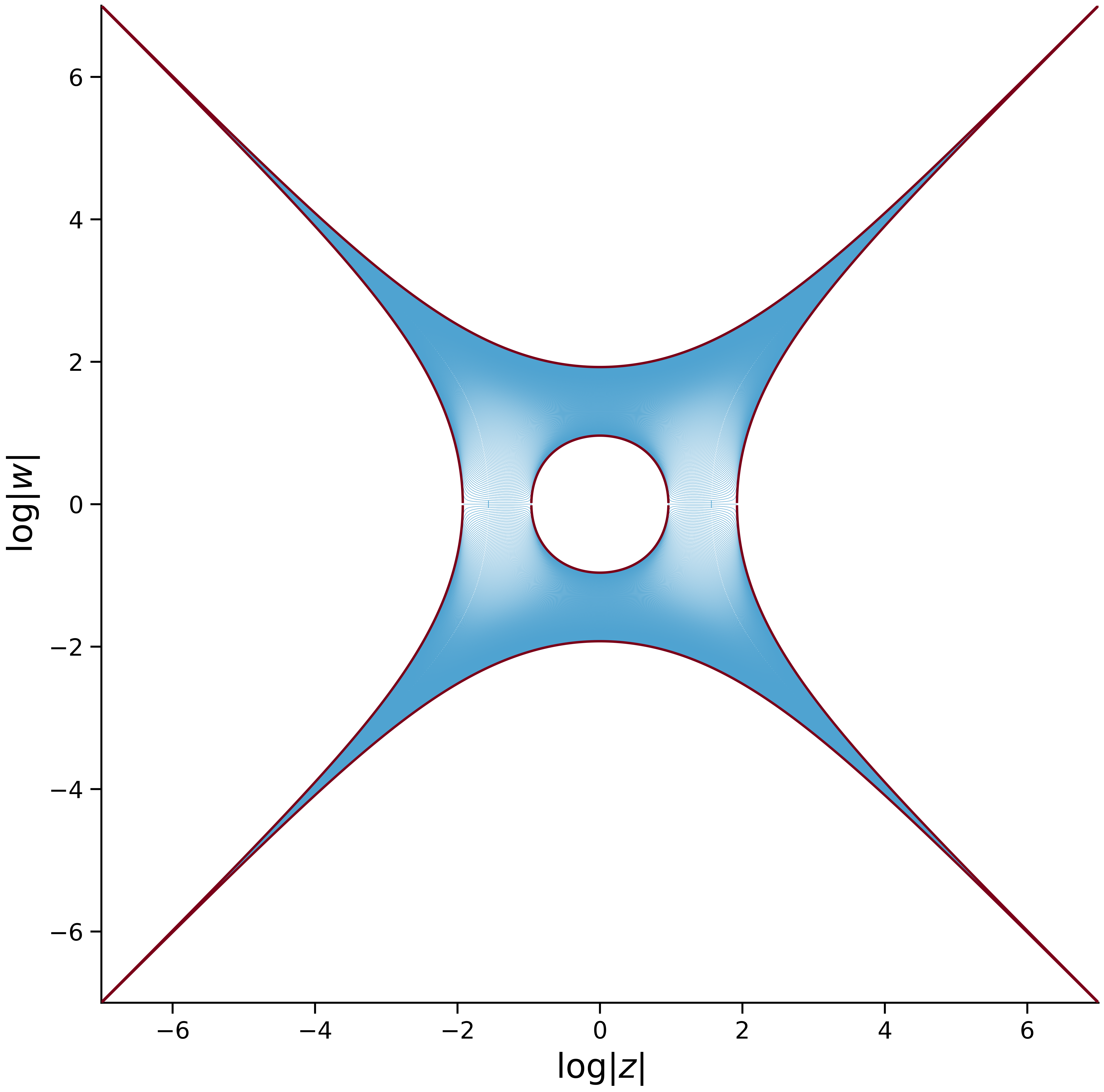}
\caption{Genus-one amoeba in light blue and its complete boundary-only contour in dark red.}
\label{fig:genus-one-control}
\end{figure}


\section{Case of irreducible curves of higher genus}

The amoeba of an algebraic curve records the moduli of its torus coordinates, while its contour records the critical values of the logarithmic map.  These two objects must be distinguished because every amoeba-boundary point is a critical value, whereas a critical value may remain inside the amoeba when another logarithmic sheet covers both sides of it.  The convexity theorem for amoeba complements and the order map are due to Forsberg, Passare, and Tsikh \cite{FPT}; the logarithmic Gauss-map description of the critical locus is part of the standard amoeba framework developed by Mikhalkin and extended in higher codimension by Madani and Nisse \cite{Mikhalkin,MadaniNisse}.  The construction below gives a complete exact instance of this boundary--contour distinction.

Consider the Laurent polynomial $F(z,w)=(z-10)^2+(zw-10)^2-1$.  Expanded in the standard coordinates, it is $F(z,w)=z^2w^2-20zw+z^2-20z+199$.  Thus the curve considered below is
$$
C=V_F=\{(z,w)\in(\C^*)^2:F(z,w)=0\}.
$$
This is not a circle or an ellipse in the affine $(z,w)$-plane: it is a plane quartic because its total degree is four, and its defining polynomial contains the term $z^2w^2$.  The curve is nevertheless accessible to exact analysis because the monomial map $T:(\C^*)^2\longrightarrow(\C^*)^2$ given by $T(z,w)=(x,y)=(z,zw)$ is an algebraic automorphism.  Its inverse is $T^{-1}(x,y)=(x,y/x)$.

The amoeba of $C$ is $\A_F=\Log(C)$, where $\Log(z,w)=(\log|z|,\log|w|)$.  Its complete logarithmic contour is $\mathcal C_F=\Log(\Crit(\Log|_C))$.  The word complete is important: $\mathcal C_F$ contains every critical value, including critical values lying in the interior of $\A_F$.

\begin{theorem}\label{thm:main}
The curve $C$ is smooth, connected and irreducible.  Its logarithmic critical locus is smooth, its complete contour $\mathcal C_F$ is a smooth embedded real-analytic one-manifold, and $\partial\A_F$ is strictly contained in $\mathcal C_F$.  More precisely, $\mathcal C_F$ is the disjoint union of the amoeba boundary and a compact internal oval.
\end{theorem}

The mechanism in Theorem~\ref{thm:main} differs visibly from the original circle and ellipse equations.  It is also useful to record its limitation: the quartic is monomially equivalent inside the algebraic torus to a translated circle.  Consequently the example is a higher-degree plane equation with a different Newton polygon, while the proof transfers the established contour geometry through a torus automorphism.  The final part of the article passes from this unimodular transformation to finite nonunimodular monomial coverings.  This produces a single family of positive-genus curves and replaces the three separate papers devoted to genera one, two, and three.  The genus computation uses the toric genus formula of Khovanskii \cite{Khovanskii1978}, while the critical-locus and contour arguments use the logarithmic Gauss-map framework of Mikhalkin and Madani--Nisse \cite{Mikhalkin,MadaniNisse}.

\subsection{Smoothness, connectedness and irreducibility}

Put $G(x,y)=(x-10)^2+(y-10)^2-1$.  The identity $F=G\circ T$ gives $C=T^{-1}(V_G)$.  The affine conic $V_G$ is nonsingular because $G_x=2(x-10)$ and $G_y=2(y-10)$ cannot vanish simultaneously on $G=0$.  Since $T$ is an algebraic isomorphism, $C$ is smooth.

The polynomial $G$ is irreducible over $\C$.  Indeed, a reducible polynomial of total degree two defining a nonsingular affine conic would be a product of two affine linear forms.  Their projective lines intersect, and the product would be singular at their intersection.  This contradicts the nonsingularity of the projective conic $X^2+Y^2-20XT-20YT+199T^2=0$.  The pullback by an automorphism of the Laurent polynomial ring preserves irreducibility.  Therefore $F$ is irreducible in $\C[z^{\pm1},w^{\pm1}]$, and $C$ is connected.

The same conclusion can be checked directly from the partial derivatives $F_w=2z(zw-10)$ and $F_z=2(z-10)+2w(zw-10)$.  If $F=F_z=F_w=0$ at a torus point, then $F_w=0$ and $z\ne0$ imply $zw=10$.  The equation $F_z=0$ then gives $z=10$, hence $w=1$, but $F(10,1)=-1$.  No singular torus point exists.

\subsection{The logarithmic linear conjugacy}

Write $(X,Y)=\Log(z,w)$ and $(U,V)=\Log(x,y)$.  Since $x=z$ and $y=zw$, one has $(U,V)=(X,X+Y)$.  Let $L:\R^2\to\R^2$ be $L(X,Y)=(X,X+Y)$.  Its inverse is $L^{-1}(U,V)=(U,V-U)$, and $\det L=1$.  The identity
$$
\Log\circ T=L\circ\Log
$$
implies $\A_F=L^{-1}(\A_G)$ and $\partial\A_F=L^{-1}(\partial\A_G)$.

Differentiation gives $d\Log_{T(p)}\circ dT_p=L\circ d\Log_p$.  Both $dT_p$ and $L$ are invertible.  Hence $p$ is critical for $\Log|_C$ precisely when $T(p)$ is critical for $\Log|_{V_G}$.  It follows that $\Crit(\Log|_C)=T^{-1}(\Crit(\Log|_{V_G}))$ and
$$
\mathcal C_F=L^{-1}(\mathcal C_G).
$$
Every regularity, embeddedness, disjointness and boundary-versus-interior assertion for the logarithmic contour of $V_G$ is therefore preserved by the invertible real linear map $L^{-1}$.

\subsection{Exact critical-value parametrization}

The conic $V_G$ has the rational parametrization $x(t)=10+(1-t^2)/(1+t^2)$ and $y(t)=10+2t/(1+t^2)$ for $t\in\mathbb P^1(\C)$.  The corresponding parametrization of $C$ is
$$
z(t)=\frac{11+9t^2}{1+t^2},\qquad
w(t)=\frac{10t^2+2t+10}{11+9t^2}.
$$
The logarithmic Gauss map of $G$ is $\gamma_G=[xG_x:yG_y]=[x(x-10):y(y-10)]$.  The equality between the logarithmic critical locus and $\gamma_G^{-1}(\mathbb P^1(\R))$ is the plane-curve logarithmic Gauss criterion \cite{Mikhalkin,MadaniNisse}.  Substitution of the rational parametrization gives
$$
\gamma_G(t)=\left[(11+9t^2)(1-t^2):2t(10t^2+2t+10)\right].
$$
The critical locus is $\gamma_G^{-1}(\mathbb P^1(\R))$.  If $\lambda\in\R$ denotes the affine Gauss value, every finite point of this locus is obtained from a root of
$$
H_\lambda(t)=(11+9t^2)(1-t^2)-2\lambda t(10t^2+2t+10)=0.
$$
The corresponding critical value for the quartic is
$$
\left(\log\left|\frac{11+9t^2}{1+t^2}\right|,
\log\left|\frac{10t^2+2t+10}{11+9t^2}\right|\right).
$$
These two equations generate the entire red contour in Figure~\ref{fig:amoeba}.  No branch is discarded.  Values at $\lambda=\infty$ are recovered from the roots of $2t(10t^2+2t+10)=0$, and the points $t=\pm i$ describe the toric ends.

For completeness, the absence of a contour cusp or transverse self-intersection can be verified without relying on the drawing.  Differentiate $H_\lambda(t)=0$ along every simple root and substitute the derivative into the two logarithmic coordinates above.  Clearing the nonzero factors $t$, $1+t^2$, $11+9t^2$ and $10t^2+2t+10$ reduces simultaneous vanishing of both coordinate derivatives to $81|t^2+1|^4+20|9t^2+11|^2+20|10t^2+2t+10|^2=0$.  The three summands are nonnegative and cannot vanish simultaneously.  Thus every critical-value branch is immersed.  Applying the same elimination to two parameters $t\ne s$ having equal logarithmic moduli reduces, after saturation by $t-s$ and by the toric factors, to the sum $|t-s|^2\{81|1+ts|^2+20|9ts+11|^2+20|10ts+t+s+10|^2\}=0$.  Hence two distinct critical branches have no finite common image.  The remaining limiting points belong to distinct toric directions.  The contour is therefore a smooth embedded real-analytic one-manifold.

\subsection{Why the inclusion is strict}

The real oval $x=10+\cos\theta$, $y=10+\sin\theta$ lies in the critical locus because its logarithmic Gauss ratio is real.  Its image on $C$ is
$$
\eta(\theta)=\left(\log(10+\cos\theta),
\log(10+\sin\theta)-\log(10+\cos\theta)\right).
$$
This is the small closed red oval visible in Figure~\ref{fig:amoeba}.  It is not an amoeba-boundary component.  It suffices to exhibit one regular point of the logarithmic map above one value of this oval.  Consider the value corresponding to the real point $(x,y)=(11,10)$.  Besides that critical preimage, the same logarithmic fiber contains a nonreal point of the form $x=11e^{i\alpha}$, $y=10e^{i\beta}$ with
$$
-1.568063<\alpha<-1.568061,\qquad 1.673469<\beta<1.673472.
$$
Substitution in $(x-10)^2+(y-10)^2=1$ isolates a unique solution in this rectangle.  At its midpoint the enclosure is $x=0.0300762-10.9999589i$ and $y=-1.0249363+9.9473366i$.  The Jacobian determinant of the real and imaginary parts of $G(11e^{i\alpha},10e^{i\beta})$ with respect to $(\alpha,\beta)$ is enclosed in $(96490,96493)$.  It is nonzero.  The inverse-function theorem then supplies an open neighborhood of $(\log11,\log10)$ contained in $\A_G$.

Thus the critical value $(\log11,\log10)$ lies in $\operatorname{Int}(\A_G)$.  The internal oval is connected and disjoint from the amoeba boundary, so its whole image lies in the interior.  Applying $L^{-1}$ shows that $(\log11,\log(10/11))$ is an interior point of $\A_F$ lying on $\mathcal C_F$.  Consequently $\partial\A_F\subsetneq\mathcal C_F$.

The blue color must occur on both sides of this internal oval.  An uncolored center would be a sampling artifact.  The Newton polygon computed below has no interior lattice point, so the order-map theorem also rules out a bounded complement component for this amoeba \cite{FPT,Rullgard}.

\subsection{Amoeba and complete contour}

In Figure~\ref{fig:amoeba}, dark blue represents the amoeba and dark red represents the complete critical-value set.  The red noncompact arcs form the amoeba boundary, while the small red oval is internal.  The image was computed from the exact rational parametrization, and the oval interior was filled explicitly so that numerical sampling does not create a false white component.

\begin{figure}[ht]
\centering
\includegraphics[width=.42\textwidth]{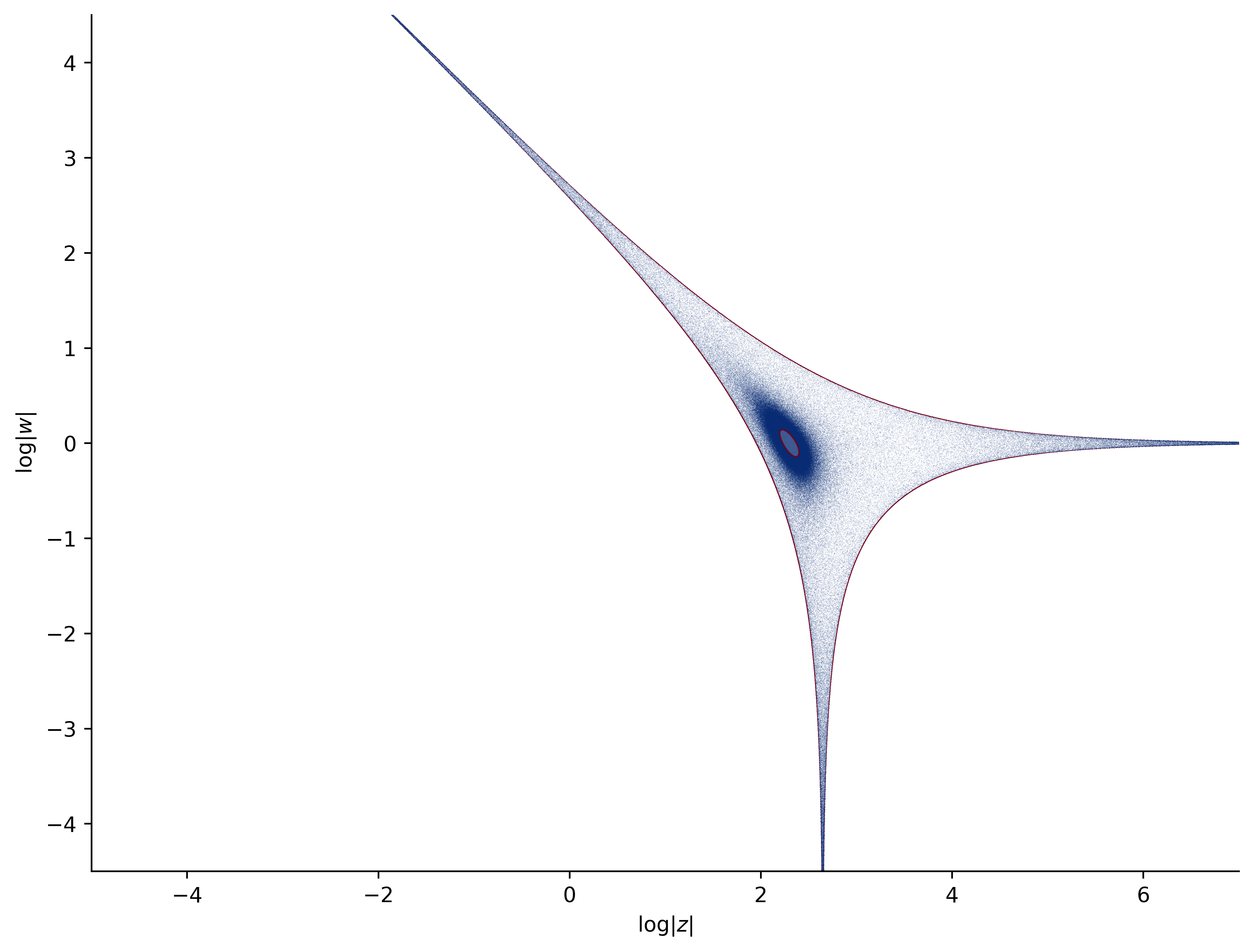}
\caption{The amoeba $\A_F$ in dark blue and its complete logarithmic contour $\mathcal C_F$ in dark red.}
\label{fig:amoeba}
\end{figure}

The exponent support is $\{(0,0),(1,0),(2,0),(1,1),(2,2)\}$.  Hence the Newton polygon is the triangle $\Newt(F)=\operatorname{conv}\{(0,0),(2,0),(2,2)\}$.  The coefficients attached to the five support points, in the displayed order, are $199,-20,1,-20,1$.

\begin{figure}[ht]
\centering
\begin{tikzpicture}[scale=1.5]
\fill[amoebablue!14] (0,0)--(2,0)--(2,2)--cycle;
\draw[amoebablue,very thick] (0,0)--(2,0)--(2,2)--cycle;
\foreach \x/\y/\a/\dx/\dy in {0/0/199/.06/.10,1/0/{-20}/.05/.10,2/0/1/.05/.10,1/1/{-20}/.07/.08,2/2/1/.05/.10}{
  \fill[contourred] (\x,\y) circle (1.7pt);
  \node[anchor=west] at (\x+\dx,\y+\dy) {$\a$};
}
\draw[->] (-.15,0)--(2.3,0) node[right] {$j$};
\draw[->] (0,-.15)--(0,2.3) node[above] {$k$};
\end{tikzpicture}
\caption{The Newton polygon of $F$, with every support point labeled by its coefficient.}
\label{fig:newton}
\end{figure}
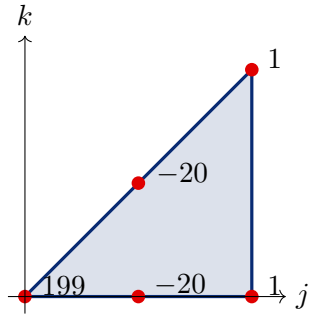

The Euclidean area of this triangle is $2$, and Pick's formula shows that it has no interior lattice point.  Its three edges account for the three asymptotic directions of the amoeba.  The monomial transformation sends the Newton triangle of $G$ by the integral exponent map $(a,b)\mapsto(a+b,b)$, which explains the sheared triangular polygon in Figure~\ref{fig:newton}.

\section{A uniform construction in every positive genus}

The quartic above has geometric genus zero because it is obtained from the translated conic by a torus automorphism.  Positive genus is obtained by replacing the unimodular automorphism by a finite monomial covering.  For every integer $n\geq2$, define
$$
F_n(z,w)=(z^n-10)^2+(w-10)^2-1
=z^{2n}-20z^n+w^2-20w+199
$$
and write $C_n=V_{F_n}\subset(\C^*)^2$.  Let $\A_n=\Log(C_n)$, let $S_n=\Crit(\Log|_{C_n})$, and let $\mathcal C_n=\Log(S_n)$.

\begin{theorem}[Uniform all-genera construction]\label{thm:all-genera}
For every integer $n\geq2$, the curve $C_n$ is smooth, connected, irreducible, and Newton nondegenerate.  Its smooth toric compactification has genus $g=n-1$.  The logarithmic critical locus $S_n$ is smooth, the complete contour $\mathcal C_n$ is a smooth embedded real-analytic one-manifold, and
$$
\partial\A_n\subsetneq\mathcal C_n.
$$
The contour is the disjoint union of its boundary-producing noncompact branches and a compact oval contained in $\operatorname{Int}(\A_n)$.  Moreover, $C_n$ is not torically equivalent to a circle, an ellipse, or any smooth conic.  Consequently every prescribed genus $g\geq1$ occurs by taking $n=g+1$.
\end{theorem}

\begin{proof}
Put $G(u,v)=(u-10)^2+(v-10)^2-1$ and consider the monomial map $\pi_n:(\C^*)^2\to(\C^*)^2$ defined by $\pi_n(z,w)=(z^n,w)$.  Then $F_n=G\circ\pi_n$.  The Jacobian determinant of $\pi_n$ equals $nz^{n-1}$ and never vanishes on $(\C^*)^2$.  Thus $\pi_n$ is a finite unramified covering of degree $n$.

Let $L_n:\R^2\to\R^2$ be $L_n(X,Y)=(nX,Y)$.  The equality $\Log\circ\pi_n=L_n\circ\Log$ gives
$$
\A_n=L_n^{-1}(\A_G),\qquad
\partial\A_n=L_n^{-1}(\partial\A_G).
$$
Differentiating this identity at $p\in C_n$ gives $d\Log_{\pi_n(p)}\circ d(\pi_n)_p=L_n\circ d\Log_p$.  Since both $d(\pi_n)_p$ and $L_n$ are invertible, $p$ is logarithmically critical precisely when $\pi_n(p)$ is logarithmically critical on $V_G$.  Consequently $S_n=\pi_n^{-1}(S_G)$ and $\mathcal C_n=L_n^{-1}(\mathcal C_G)$.  The contour of the translated circle is a smooth embedded set consisting of its boundary-producing branches and one compact internal oval.  The invertible map $L_n^{-1}$ preserves embeddedness, disjointness, real analyticity, and the distinction between boundary and interior.  It follows that $\mathcal C_n$ is smooth and embedded and that $\partial\A_n\subsetneq\mathcal C_n$.

The derivatives are $(F_n)_z=2nz^{n-1}(z^n-10)$ and $(F_n)_w=2(w-10)$.  If a torus point satisfied $F_n=(F_n)_z=(F_n)_w=0$, then $z^n=10$ and $w=10$, whereas $F_n=-1$ there.  Hence $C_n$ is smooth.

Writing $y=w-10$, the equation becomes $y^2=-(z^n-9)(z^n-11)$.  The polynomial on the right has $2n$ distinct zeros and is not a square in $\C(z)$.  The equation is therefore irreducible over $\C(z)$, and Gauss's lemma implies that $F_n$ is irreducible in the Laurent polynomial ring.  A smooth irreducible complex curve is connected.

The Newton polygon is $\Delta_n=\operatorname{conv}\{(0,0),(2n,0),(0,2)\}$.  The horizontal face polynomial $(z^n-9)(z^n-11)$ and the vertical face polynomial $(w-9)(w-11)$ have only simple roots.  The diagonal face polynomial is $z^{2n}+w^2$, whose partial derivatives cannot vanish simultaneously in the edge torus.  Therefore $F_n$ is Newton nondegenerate.

The area of $\Delta_n$ is $2n$.  Its boundary lattice count is $2n+4$, because the three edge lattice lengths are $2n$, $2$, and $\gcd(2n,2)=2$.  Pick's formula gives
$$
\#\bigl(\operatorname{Int}(\Delta_n)\cap\mathbb Z^2\bigr)
=2n-\frac{2n+4}{2}+1=n-1.
$$
The interior lattice points are exactly $(j,1)$ for $1\leq j\leq n-1$.  Newton nondegeneracy and the toric genus formula therefore give genus $n-1$ \cite{Khovanskii1978}.

The compact critical oval of $V_G$ is parametrized by $(u,v)=(10+\cos\theta,10+\sin\theta)$.  Pulling its logarithmic image back through $L_n$ gives
$$
\eta_n(\theta)=\left(\frac1n\log(10+\cos\theta),\log(10+\sin\theta)\right).
$$
The regular nonreal sheet that covers a neighborhood on both sides of the oval for $V_G$ also pulls back through the unramified map $\pi_n$.  Hence amoeba points occur on both sides of $\eta_n$, so $\eta_n(S^1)\subset\operatorname{Int}(\A_n)$.

The differential identity proves absence of ramification of $\Log|_{S_n}$.  If a nonzero tangent vector to $S_n$ were annihilated by $d\Log$, its image under $d\pi_n$ would be a nonzero tangent vector to $S_G$ annihilated by the logarithmic differential, contradicting the established immersion of the translated-circle critical locus.

For embeddedness, let $\zeta_n=e^{2\pi i/n}$.  The deck transformations are $\iota_k(z,w)=(\zeta_n^kz,w)$.  Deck-related critical points have the same logarithmic image germ because $\Log\circ\iota_k=\Log$; the corresponding deck graphs must be removed in the normalized two-lift system.  If two critical points not related by a deck transformation had the same logarithmic value, their images under $\pi_n$ would give two distinct critical lifts with the same logarithmic value on $V_G$, contradicting embeddedness of $\mathcal C_G$.  Thus no off-deck crossing or tangential coincidence occurs.  The simple face roots and the linear identity with $L_n$ also prevent collisions at the toric ends.  This gives uniform ramification, normalized two-lift, and toric-end certificates; the distinction between ramification and multiple critical lifts is discussed in \cite{NisseSing}.

Finally, toric automorphisms preserve the function field, normalization, and geometric genus.  Every $C_n$ has genus $n-1>0$, whereas a smooth irreducible conic has genus zero.  Hence $C_n$ is not torically equivalent to a circle, an ellipse, or any smooth conic.
\end{proof}

The covering produces $n$ source lifts of each critical component, but it does not produce $n$ distinct contour-image ovals.  All points in one deck orbit have the same logarithmic image.  The cases $n=2$, $n=3$, and $n=4$ are exactly the former genus-one, genus-two, and genus-three constructions, so they no longer require separate statements or proofs.

\subsection{A genus-four illustration of the uniform family}

For $n=5$, the curve $F_5(z,w)=z^{10}-20z^5+w^2-20w+199=0$ has genus four.  Figure~\ref{fig:family-amoeba} displays its amoeba in blue and its complete logarithmic contour in dark red.  The small red oval is internal, with amoeba points on both sides, and the noncompact red branches contain the boundary-producing arcs.

\begin{figure}[ht]
\centering
\includegraphics[width=.54\textwidth]{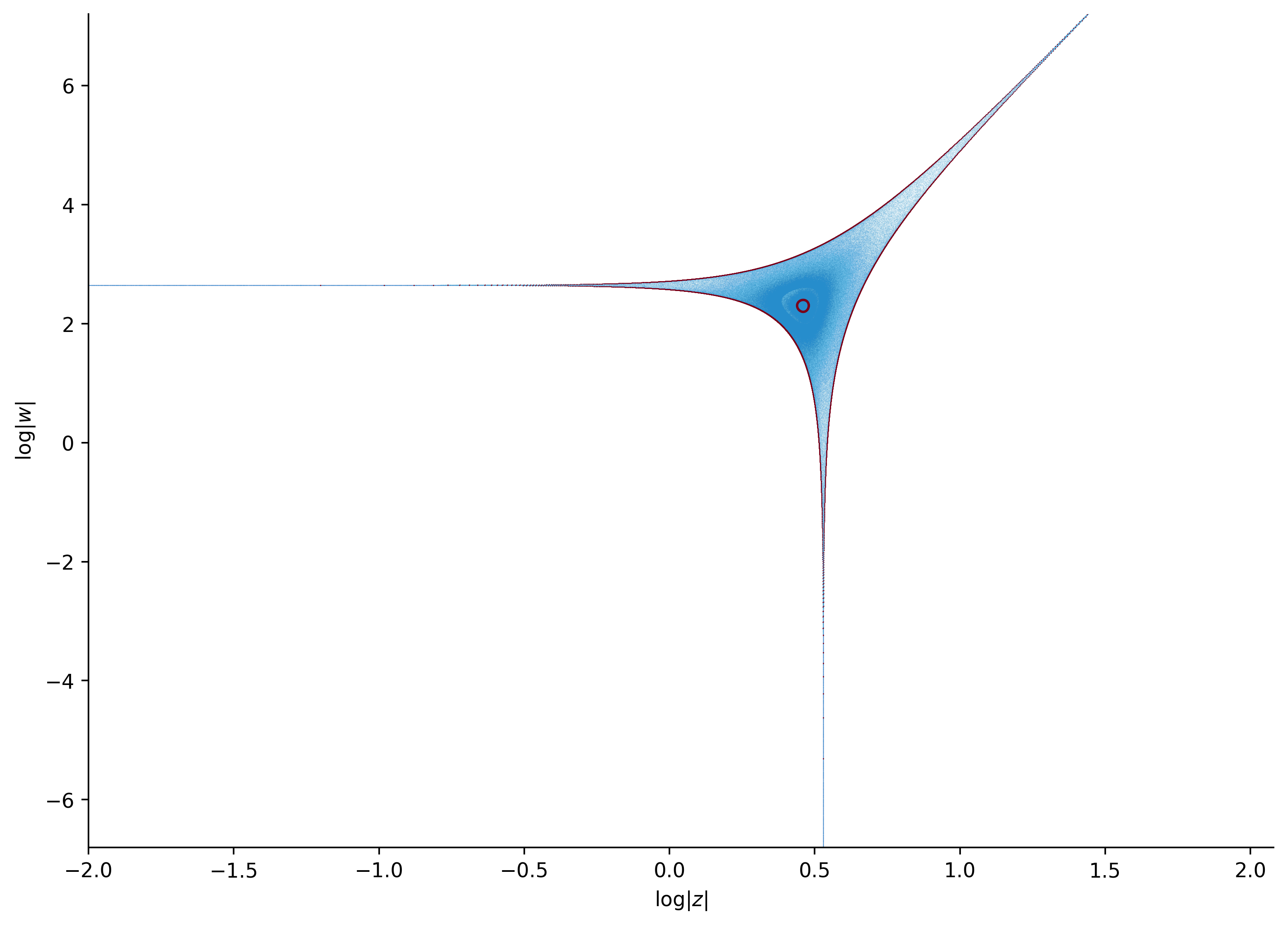}
\caption{The amoeba of $F_5$ in blue and its complete contour in dark red.  The compact oval is contained in the amoeba interior.}
\label{fig:family-amoeba}
\end{figure}

The Newton polygon is $\Delta_5=\operatorname{conv}\{(0,0),(10,0),(0,2)\}$.  Its four interior lattice points are $(1,1)$, $(2,1)$, $(3,1)$, and $(4,1)$, in agreement with the genus computation.

\begin{figure}[ht]
\centering
\includegraphics[width=.65\textwidth]{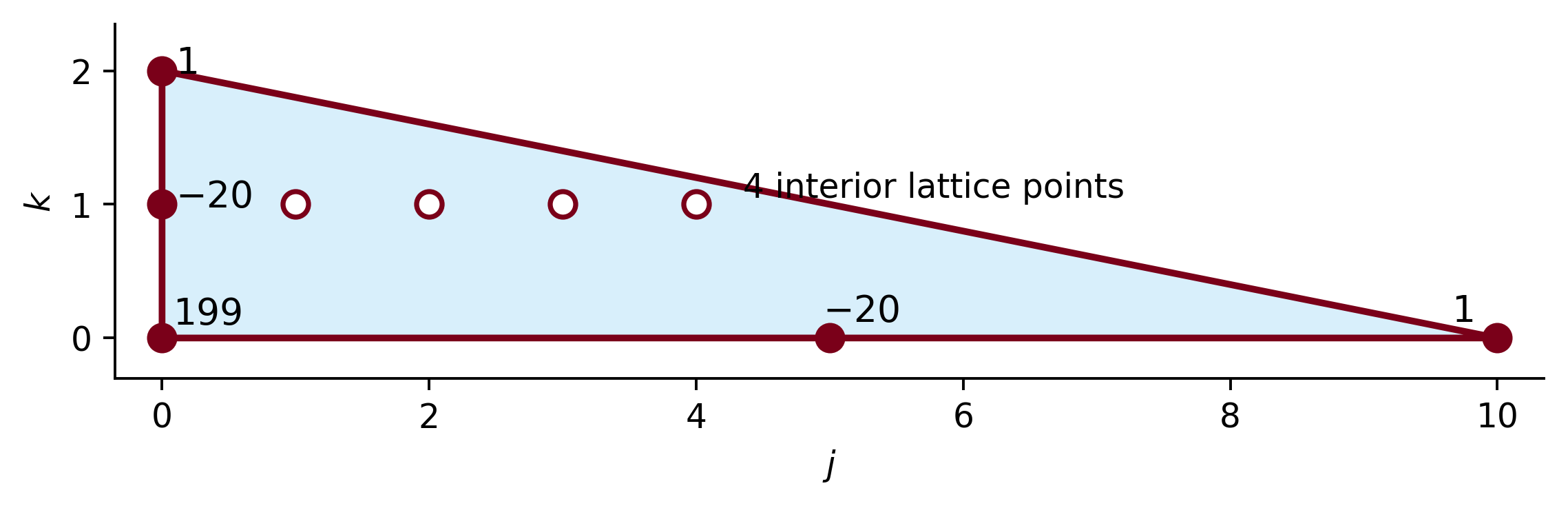}
\caption{The Newton polygon of $F_5$.  The four interior lattice points determine the genus.}
\label{fig:family-newton}
\end{figure}

We conclude that the quartic $z^2w^2-20zw+z^2-20z+199=0$ supplies a smooth connected irreducible plane curve with a smooth embedded logarithmic contour strictly larger than the amoeba boundary.  Unlike the circle and ellipse equations, it has total degree four and Newton polygon $\operatorname{conv}\{(0,0),(2,0),(2,2)\}$.  Its compact red oval is an internal critical-value component, not the boundary of a bounded complement component.  The uniform family $F_n(z,w)=(z^n-10)^2+(w-10)^2-1$ extends the same phenomenon to every positive genus.  It replaces the separate genus-one, genus-two, and genus-three constructions and identifies them as the first three members of one finite-covering theorem.


\end{document}